\documentclass[11pt, a4paper, oneside]{article}

\usepackage{graphicx} 
\usepackage{authblk}
\usepackage[margin=1in]{geometry}
\usepackage{amsmath, amssymb, mathtools, amsthm, enumerate, subcaption, tikz, algorithm, url}
\usepackage[noend]{algpseudocode}
\usepackage{xcolor}
\usepackage{natbib}

\newtheorem{Lemma}{Lemma}
\newtheorem{Corollary}{Corollary}
\newtheorem{Algorithm}{Algorithm}
\newtheorem{Theorem}{Theorem}
\newtheorem{Proposition}{Proposition}
\newtheorem{Remark}{Remark}

\setcitestyle{square}

\numberwithin{equation}{section}

\title{Smoothed Circulant Embedding with Applications to Multilevel Monte Carlo Methods for PDEs with Random Coefficients}
\author{Anastasia Istratuca\thanks{\noindent School of Mathematics and Maxwell Institute for Mathematical Sciences, University of Edinburgh,
King’s Buildings, Edinburgh EH9 3FD, UK. Email: \texttt{A.Istratuca@sms.ed.ac.uk} }, \, Aretha L Teckentrup{\footnote{School of Mathematics and Maxwell Institute for Mathematical Sciences, University of Edinburgh,
King’s Buildings, Edinburgh EH9 3FD, UK. Email: \texttt{A.Teckentrup@ed.ac.uk}}}}
\date{}

\begin{document}

\newcommand{\red}{\textcolor{red}}
\newcommand{\blue}{\textcolor{blue}}
\newcommand{\green}{\textcolor{green}}

\maketitle

\begin{abstract}
We consider the computational efficiency of Monte Carlo (MC) and Multilevel Monte Carlo (MLMC) methods applied to partial differential equations with random coefficients. These arise, for example, in groundwater flow modelling, where a commonly used model for the unknown parameter is a random field. We make use of the circulant embedding procedure for sampling from the aforementioned coefficient. To improve the computational complexity of the MLMC estimator in the case of highly oscillatory random fields, we devise and implement a smoothing technique integrated into the circulant embedding method. This allows to choose the coarsest mesh on the first level of MLMC independently of the correlation length of the covariance function of the random field, leading to considerable savings in computational cost. We illustrate this with numerical experiments, where we see a saving of up to factor 5-10 in computational cost for accuracies of practical interest.
\end{abstract}


\section{Introduction}
\label{sec: introduction}

Uncertainty quantification (UQ) is a contemporary research area concerned with identifying, assessing and reducing uncertainties related to physical models, numerical algorithms, experiments and their predicted outcomes or quantities of interest. Methodical computations of uncertainties and errors in simulations, and careful exploration of how they propagate through a model and impact its expected outcome, are vital in many applications for principled risk assessments and decision making.

One of the most common approaches to modelling physical phenomena consists of partial differential equations (PDEs), which allow for computer simulations through the use of modern numerical solvers. Such models arise, for instance, in electrodynamics \citep{tovmasyan_boundary_1994}, and fluid \citep{howe_quantum_1980} and quantum mechanics \citep{herron_partial_2008}. In this case, the uncertainty arises from unknown parameters whose estimation from physical or experimental data is impractical.

The specific UQ path we follow in this paper rests on modelling the unknown PDE coefficients as random parameters, so that their uncertainty can be defined in terms of statistical quantities, such as probability density functions or covariance functions. Further, this uncertainty can be tracked through the model to the output by resolving statistical estimates of various functionals related to the solution of the PDE model. These estimates are called quantities of interests, and they can take, for example, the form of expected values or variances.

An example of where this problem arises in practice is the modelling of groundwater flow. Specifically, water resources, generally comprising of ground and surface water, must be preserved free of pollution. Thus, efficient methods for modelling and forecasting the movement of impurities through aquifers, which are used as supplies for potable water, are necessary. Such impurities can contaminate the groundwater flowing beneath earth's surface in various ways, such as carbon capture and underground storage, fracking, accidental spills or spent nuclear fuel repositories. A mathematical model for simulating groundwater flow rests on Darcy's Law \citep{de_marsily_quantitative_1986}, where the main parameter is the hydraulic conductivity, that is, the ease with which a fluid can move through porous media or fractures under a given pressure gradient. In practice, this can only be measured at a finite, usually small, number of geographical points. However, for numerical simulations, the value of this parameter is usually required at all the points in the computational domain, which constitutes the main source of uncertainty for this problem.

There is already an abundant literature investigating PDEs with random coefficients using a variety of methods. Possible approaches include the stochastic Galerkin and stochastic collocation approaches, or sparse tensor discretisations, see e.g. the surveys \citep{schwab_sparse_2011, gunzburger_stochastic_2014, cohen_approximation_2015} and the references therein. However, the computational cost of these techniques soars up when treating random field models with low regularity and short correlation length, which are of interest in this paper.

An effective approach for tackling these problems remains the Monte Carlo (MC) method. This entails generating realisations of parameter values and subsequently solving the governing equation for many such samples to approximate a specific quantity of interest. The benefits of MC are its ease of implementation, as well as the fact that the associated computational complexity does not grow with dimension \citep{kuo_lifting_2005}. Of course, these benefits are severely weakened if one takes into account that, in this context, one or more PDEs must be solved for each MC sample, the cost of which is naturally dimension-dependent. In addition, MC suffers from an extremely slow convergence rate, rendering this approach intractable.

One alternative which aims to alleviate this issue is the Multilevel Monte Carlo (MLMC) method \citep{heinrich_multilevel_2001, giles_multilevel_2008}. This involves defining multiple levels of approximation that differ in computational cost. In the context of solving PDEs with random coefficients, the MLMC levels can be defined by different grid resolutions to solve the governing equation. Thereafter, we cluster much of the computational effort in cheap estimations on a coarse mesh with low accuracy. To increase the quality of the approximation, we compute ``correction terms'' on finer levels with progressively higher precision. Since most of the uncertainty can be captured on the coarsest grid, comparatively fewer samples are required on the subsequent finer levels to obtain a sufficiently good estimate of the quantity of interest. This architecture of levels, which divides the work required to achieve a certain accuracy, leads to a significant reduction in the overall computational complexity of MLMC compared to standard MC.

One technique which is widely used for sampling from a random field is the Karhunen-Lo\`eve (KL) expansion \citep{kosambi_statistics_1943}. This has been extensively studied both theoretically and numerically for PDEs with log-normal random input with MLMC \citep{cliffe_multilevel_2011, teckentrup_multilevel_2012, charrier_strong_2012, teckentrup_further_2013, gittelson_multi-level_2013, charrier_finite_2013, giles_multilevel_2013, schwab_multilevel_2023}, as well as with Quasi-Monte Carlo methods (QMC) and Multilevel Quasi-Monte Carlo methods (MLQMC), see, for example, the survey \citep{kuo_application_2018} and the references therein. Further alterations of this algorithm to accommodate the difficulties arising from random field models with short correlation length have been proposed in \citep{teckentrup_further_2013, gittelson_multi-level_2013, schwab_multilevel_2023}. Other sampling methods have also been studied in the context of MLMC or MLQMC for PDEs with random coefficients, such as wavelet reconstruction approach \citep{bachmayr_representations_2018} or the white noise sampling method \citep{croci_multilevel_2021}.

In this paper, we focus on the circulant embedding technique for sampling from the random parameter. While this offers a discrete representation of a random field on a given mesh, it is both exact at these grid points, and computationally efficient. This is in comparison with the KL-expansion, which, once truncated, offers a continuous approximation whose rate of convergence is intrinsically linked to the decay of the eigenvalues of the covariance function's kernel, which has a direct impact on the computational cost. {Nevertheless, the KL-expansion and the circulant embedding method share certain analytical properties (as discussed further in section \ref{sec: eigenvalues-decay}).} The circulant embedding method has been integrated with MLMC and MLMQMC in e.g. \citep{graham_quasi-monte_2011, kuo_multilevel_2017}. 

The main novelty in this paper is adjusting the circulant embedding method to produce smooth approximations on coarse meshes of samples from extremely oscillatory random fields, i.e. with small correlation length, high variance and low regularity. For the application in groundwater flow outlined above, this problem usually arises when we consider large computational domains where the correlation length is considerably smaller compared to the size of the domain. This is of particular interest in MLMC for PDEs with random coefficients, as the computational gains prompted by MLMC rest on using coarse meshes, which is usually not feasible for the problem at hand. A similar algorithm has been proposed in \citep{sawko_effective_2022}, albeit for random fields with large-scale fluctuations which still require a fine grid representation. In particular Sawko et al. focus on deriving low rank approximations of the covariance matrix of the random field, and do not provide a theoretical analysis of the corresponding error.

The novel contributions of this work can be summarised as follows:
\begin{itemize}
    \item We adapt the circulant embedding method to sample from smoothed versions of a Gaussian field with short correlation length, and prove a bound on the introduced smoothing error. This method is not specific to applications in MLMC, but is a general method for obtaining smoothed realisations of highly oscillatory Gaussian fields which can accurately be represented on coarse grids.
    \item We integrate the smoothed circulant embedding method into MLMC, derive an optimal coupling between the smoothing and grid sizes on each level, and prove a complexity theorem for the resulting MLMC estimator.
    \item We demonstrate the gains in computational cost prompted by introducing smoothing into MLMC in a simple test problem in groundwater flow, where we observe a saving of factor 5-10 {in the settings where the performance of MLMC without smoothing is severely limited by the highly oscillatory nature of the PDE coefficient.}
\end{itemize}

The remainder of this manuscript is organised as follows. Section \ref{sec: ces} offers an overview of circulant embedding (CE) methods, as well as the proposed smoothing technique (CES), together with the corresponding error analysis. Further, Section \ref{sec: mlmc} outlines key aspects of the Monte Carlo and Multilevel Monte Carlo approaches to uncertainty quantification, with a focus on their application to PDEs with random input and their integration with CE and CES. The paper concludes with Section \ref{sec: numerical-results}, which summarises numerical experiments aimed at illustrating the potential of the smoothing procedure, particularly when combined with MLMC. We provide a short discussion and conclusion in Section \ref{sec: conclusions}.

\section{Circulant Embedding with Smoothing}
\label{sec: ces}

In this section, we begin with a brief review of the standard circulant embedding method for sampling from a stationary Gaussian random field. This relies on embedding the covariance matrix of the random field in a larger matrix which has a block circulant structure \citep{dietrich_fast_1993}, and is computationally very cheap because it exploits the Fast Fourier Transform (FFT) algorithm \citep{frigo_design_2005} for factorising the embedding matrix. We then briefly outline the smooth periodisation approach for ensuring the positive definiteness of the embedding matrix \citep{bachmayr_representations_2018}. Finally, we describe the smoothing technique for highly oscillatory random fields.

For simplicity, we shall write $a \lesssim b$ if $a/b$ is uniformly bounded independent of any parameters for two positive quantities $a$ and $b$. In addition, we write $a \simeq b$ if $a \lesssim b$ and $b \lesssim a$.

\subsection{Circulant Embedding}
\label{sec: ce}

Suppose we are interested in sampling from a stationary Gaussian field $Z$ on a domain $D \subset \mathbb{R}^d$ and probability space $(\Omega, \mathcal{F}, \mathcal{P})$ with continuous covariance function $r(\mathbf{x}, \mathbf{y})$, so that:
\begin{equation}
\label{eq: rf-mean-cov}
    \mathbb{E}[Z(\mathbf{x}, \cdot)] = 0 \quad \text{and} \quad \mathbb{E}[Z(\mathbf{x}, \cdot) Z(\mathbf{y}, \cdot)] = r(\mathbf{x}, \mathbf{y}) = C(\mathbf{x} - \mathbf{y}), \quad \forall \, \mathbf{x}, \mathbf{y} \in D,
\end{equation}
where $C: \mathbb{R}^d \rightarrow \mathbb{R}$ is a suitably chosen function. For simplicity, we set the mean of $Z$ to $0$ throughout this work, but we note that everything can be easily extended to non-zero mean $m(\mathbf{x}):D \rightarrow \mathbb R$. {Although the circulant embedding methodology is generally applicable, in the analysis and numerical examples presented in this work we focus on using either a Mat\'ern or separable exponential covariance function. The Mat\'ern covariance function is given by:
\begin{equation}
\label{eq: matern-cov-function}
    C(\mathbf{t}) = \sigma^2 \frac{2^{1-\nu}}{\Gamma(\nu)} \left(\frac{\sqrt{2\nu}\|\mathbf{t}\|_2}{\lambda}\right)^\nu K_\nu \left(\frac{\sqrt{2\nu}\|\mathbf{t}\|_2}{\lambda}\right),
\end{equation}
where $\nu > 0$ is the smoothness exponent, $\lambda > 0$ denotes the length scale, $\sigma^2 > 0$ is the marginal variance and $K_\nu$ stands for the modified Bessel function of the second kind. The separable exponential covariance {function}, on the other hand, is given by:
\begin{equation}
\label{eq: p-norm-cov-function}
    C(\mathbf{t}) \coloneqq \sigma^2 \exp\left(\frac{-\|\mathbf{t}\|_p}{\lambda}\right),
\end{equation}
with $p=1$. Here, $\| \cdot \|_p$ denotes the $\ell_p$ norm on $\mathbb{R}^d$, and the parameters $\sigma^2>0$ and $\lambda>0$ stand for the marginal variance and the length scale of $C$, respectively.}

Further, given a zero-mean Gaussian vector $\mathbf{Z} \in \mathbb{R}^M$ representing the field $Z$ at $M$ locations in $D$, with positive definite covariance matrix $R \in \mathbb{R}^{M \times M}$ so that $R = \mathbb{E}[\mathbf{Z}\mathbf{Z}^T]$, a sample from this discrete representation of the random field $Z$ can be generated from any factorisation of its covariance matrix $R$ of the form:
\begin{equation*}
    R = \Theta \Theta^T, \quad \Theta \in \mathbb{R}^{M \times M}.
\end{equation*}
In fact, for any vector $\boldsymbol{\xi} \coloneqq \left(\xi_1(\omega), \, \xi_2(\omega), \dotsc, \xi_M(\omega)\right)^T$ of independent standard Gaussian random variables, we can simply take:
\begin{equation}
\label{eq: sample-from-Z}
    \mathbf{Z} \coloneqq \Theta \boldsymbol{\xi}.
\end{equation}
This defines a suitable realisation from the discrete representation of the random field $Z$, since we can easily check that Eq. \eqref{eq: sample-from-Z} implies $\mathbb{E}[\mathbf{Z} \mathbf{Z}^T] = R$, as required.

There are multiple techniques to factorising the covariance matrix $R$ of the random field $Z$, such as the Cholesky decomposition. However, such commonly used algorithms have computational complexity $\mathcal{O}({M}^3)$ \citep{higham_accuracy_2002}, {where, for example, $M = m^d$ for a uniform grid of the domain $D$ with mesh size $m$ in all coordinate directions}. On the other hand, the circulant embedding method provides a fast approach to factorising $R$ by first embedding it in a larger matrix with circulant structure, which can then be factorised using the FFT. This has complexity $\mathcal{O}({s \log_2 s})$ \citep{davis_circulant_1979}, {where $s = (2m)^d$, with $m$ as in the previous example}. To describe this procedure, let us consider the one- and two-dimensional cases individually. Higher dimensional cases follow analogously. 

\subsubsection{One-Dimensional Case}
\label{sec: sec-3-1d-case}

Suppose now that $Z$ is a stationary line Gaussian process, with zero mean and covariance function $r(x,y) \coloneqq C(x-y)$, as described in section \ref{sec: ce}. Further, suppose that the domain is $D = [0,1]$, which is discretised using a grid $\mathcal{T}$ consisting of $m+1$ equispaced sampling points:
\begin{equation*}
    \mathcal{T} = \{ x_k = k \Delta x\}_{k=0}^m = \left\{x_k = \frac{k}{m}\right\}_{k=0}^m,
\end{equation*}
where $\Delta x = \frac{1}{m}$ is the mesh size.
The advantage of considering equidistant points is that it allows us to characterise the entries in the covariance matrix $R$ as follows:
\begin{equation*}
    R_{j,k} = C(x_j-x_k) = C(|j-k| \Delta x) = C\left(\frac{|j-k|}{m}\right) \eqqcolon C_{|j-k|},
\end{equation*}
as $C((j-k) \Delta x) = C((k-j) \Delta x) = C(|j-k| \Delta x)$, $\forall j, k = 0, \dotsc m$.
Thus, the covariance matrix $R$ has the following form:
\begin{equation*}
    R = \begin{bmatrix}
    C_0 & C_1 & C_2 & \dotsc & C_m \\
    C_1 & C_0 & C_1 & \dotsc & C_{m-1} \\
    C_2 & C_1 & C_0 & \dotsc & C_{m-2} \\
    \vdots & \vdots & \vdots & \ddots & \vdots \\
    C_m & C_{m-1} & C_{m-2} & \dotsc & C_0
    \end{bmatrix}.
\end{equation*}
In other words, $R$ is a \textit{symmetric Toeplitz} matrix, which we denote by $\text{Toeplitz}(C_0, \dotsc, C_m)$. We then embed $R$ into a larger circulant matrix $S$ by mirroring its components, i.e. $S = \text{Toeplitz}(C_0, \dotsc, C_m, C_{m-1}, \dots, C_1)$, so that $S \in \mathbb{R}^{2m \times 2m}$. 

Under the assumption that the symmetric matrix $S$ is positive definite, we can diagonalise it using the Fourier transform, yielding the following eigenvalue decomposition \citep{barnett_matrices_1990}:
\begin{equation}
\label{eq: circulant-embedding-matrix-complex-factorisation-1D}
    S = F \Lambda F^* = F^* \Lambda F.
\end{equation}
Here, $F \in \mathbb{C}^{2m \times 2m}$ is the one-dimensional unitary discrete Fourier matrix, and $F^* \in \mathbb{C}^{2m \times 2m}$ is its complex conjugate transpose, whose entries are given by:
\begin{equation}
\label{eq: Fourier-matrix-entries}
    F_{p,q} = \frac{1}{\sqrt{2m}} \exp\left(\frac{2 \pi \mathrm{i} (p-1) (q-1) }{2m} \right), \quad p,q = 1, \dotsc, 2m.
\end{equation}
Further, $\Lambda$ is the diagonal matrix of eigenvalues of $S$, which are given by:
\begin{equation*}
    {\Lambda_j} = (\sqrt{2m} F \mathbf{s})_{{j}}, \quad {j = 1, \dotsc, 2m},
\end{equation*}
where $\mathbf{s} \in \mathbb{R}^{2m}$ is the first {column} of the matrix $S$.

Returning to the problem at hand, recall that we can now factorise the circulant matrix $S$ as in Eq. \eqref{eq: circulant-embedding-matrix-complex-factorisation-1D}. In some applications, such as quasi-Monte Carlo methods for PDEs with random coefficients \citep{graham_quasi-monte_2011}, we require a real factorisation of $S$. Thus, to avoid complex values, we convert the Fourier factorisation into a Hartley transform \citep{hartley_more_1942} by combining the real and imaginary parts of $F$, denoted by $\Re (F)$ and $\Im (F)$, respectively. This yields $S = G \Lambda G^T$, where $G \in \mathbb{R}^{2m \times 2m}$ is given by $G = \Re (F) + \Im (F)$, as proven in \citep[Lemma 4]{graham_quasi-monte_2011}.

Further, once we obtain a factorisation of the form $S = G \Lambda G^T$, we can take $\boldsymbol{\xi} \sim \mathcal{N}(0, I)$ to obtain $\mathbf{Z} = (G \Lambda^{\frac{1}{2}} \boldsymbol{\xi})_R$ a sample from the random field \citep[Corollary 3]{graham_quasi-monte_2011}. Here, the notation $(G \Lambda^{\frac{1}{2}} \boldsymbol{\xi})_R$ represents the entries in this vector that correspond to the location of the matrix $R$ in the embedding matrix $S$. In this case, it corresponds to the first $m+1$ entries in $G \Lambda^{\frac{1}{2}} \boldsymbol{\xi}$.

\subsubsection{Two-Dimensional Case}
\label{sec: sec-3-2d-case}
 
In this section, we let $Z$ be a stationary Gaussian process, with mean $0$ and covariance function $r(\mathbf{x}, \mathbf{y}) \coloneqq C(\mathbf{x} - \mathbf{y})$, as characterised in Section \ref{sec: ce}. In two dimensions, the circulant embedding method can be extended to a rectangular domain $D$ discretised uniformly. For simplicity, let us assume that the domain is the unit square $D = [0,1]^2$, and that its corresponding two-dimensional grid $\mathcal{T}$ is formed by $(m_1+1) \times (m_2+1)$ points of the form:
\begin{equation*}
    \mathbf{x}_{p_1,p_2} = (x_{p_1}, y_{p_2}),
\end{equation*}
where $x_{p_1} = p_1\Delta x $ and $y_{p_2} = p_2\Delta y$ with $p_i = 0, \dotsc, m_i$, $i = 1, 2$. Here, $\Delta x = \frac{1}{m_1}$ and $\Delta y = \frac{1}{m_2}$ denote the horizontal and vertical mesh sizes used in discretising $D$, respectively. 

Suppose now that the points are ordered lexicographically first in the $x$ direction, and then in the $y$ direction. Stationarity and equispatiality ensure that the entries in the covariance matrix $R$ can be expressed as follows:
\begin{equation*}
    R_{(p_1,p_2), (q_1,q_2)} = C\left(\begin{bmatrix}x_{p_1} \\ y_{p_2}\end{bmatrix} - \begin{bmatrix}x_{q_1}\\ y_{q_2}\end{bmatrix}\right) = C\left(\begin{bmatrix}\frac{|p_1-q_1|}{m_1} \\ \frac{|p_2-q_2|}{m_2}\end{bmatrix}\right) \eqqcolon C_{|p_1-q_1|, |p_2-q_2|},
\end{equation*}
using the symmetry of $C$ as in the one-dimensional case, for $p_i, q_i = 0, \dotsc, m_i$, $i=1, 2$. Thus, the covariance matrix is a \textit{symmetric block Toeplitz} matrix. Specifically, it has the following form:
\begin{equation*}
R = \text{Toeplitz}(C_{0,0}, C_{1,0}, \dotsc, C_{m_1,0}, C_{0,1}, C_{1,1}, \dotsc, C_{m_1,1}, C_{0,2} \dotsc C_{m_2, m_2}).
\end{equation*}
Hence, the covariance matrix is actually block Toeplitz and block symmetric, with blocks which are, in turn, Toeplitz and symmetric. In other words, the covariance matrix is given by:
\begin{equation*}
    R = \text{Toeplitz}(R_0, R_1, \dotsc, R_{m_2}),
\end{equation*}
where each block $R_j$, $j = 0, \dotsc, m_2$, is a Toeplitz matrix of dimension $(m_1+1) \times (m_1+1)$:
\begin{equation*}
    R_j = \text{Toeplitz}(C_{0, j}, C_{1,j}, \dotsc, C_{m_1, j}).
\end{equation*}

We can now follow a similar procedure to the method introduced in Section \ref{sec: sec-3-1d-case} for embedding $R$ into a circulant matrix. Accordingly, we first embed each block $R_j$, $j = 0, \dotsc, m_2$, into a circulant matrix $S_j$ of dimension $2m_1 \times 2m_1$. Subsequently, the resulting blocks are embedded into a large circulant matrix $S$ of dimension $4m_1 m_2 \times 4m_1 m_2$, so that $S$ becomes a block circulant matrix with circulant blocks. 
The embedding matrix $S$ can then be characterised as:
\begin{equation*}
    S = \text{Toeplitz}(S_0, S_1, \dotsc, S_{m_2}, S_{m_2-1}, \dotsc, S_1),
\end{equation*}
where each block $S_j$ is itself circulant:
\begin{equation*}
    S_j = \text{Toeplitz}(C_{0,j}, C_{1,j}, \dotsc, C_{m_1,j}, C_{m_1-1,j}, \dotsc, C_{1,j}), \quad j = 0, \dotsc, m_2.
\end{equation*}

As $S$ is a circulant matrix, it may be diagonalised via the two-dimensional discrete Fourier transform \citep{barnett_matrices_1990}. This is obtained by applying the one-dimensional Fourier transform, as described in the previous section, to each of the two coordinate directions in turn. Thus, we have:
\begin{equation*}
    S = F \Lambda F^* = F^* \Lambda F,
\end{equation*}
where $F$ is the unitary two-dimensional Fourier matrix, whose entries are given by:
\begin{equation*}
    F_{(p_1,p_2),(q_1,q_2)} = \frac{1}{\sqrt{4m_1m_2}}\exp\left(\frac{2 \pi \mathrm{i} \, (p_1-1) (q_1-1)}{2m_1}\right) \exp\left(\frac{2 \pi \mathrm{i} \, (p_2-1) (q_2-1)}{2m_2}\right),
\end{equation*}
with $p_i, q_i = 1, \dotsc, 2m_i, \, i=1,2$. Here, $F^*$ and $\Lambda$ are as in Section \ref{sec: sec-3-1d-case}. Since $S$ is circulant, it is again uniquely defined by its first {column (or row)}, and so its eigenvalues are specified by:
\begin{equation*}
    {\Lambda_j} = (\sqrt{4m_1 m_2} F \mathbf{s})_{{j}}, \quad {j = 1, \dotsc, 4m_1 m_2}, 
\end{equation*}
where $\mathbf{s}$ is the first {column} of the matrix $S$. 

In an analogous fashion to the previous section, we can modify the obtained Fourier factorisation into a Hartley transform. The corresponding argument is given in \citep[Lemma 5]{graham_quasi-monte_2011}. Similar to the one-dimensional case, $\mathbf{Z} = (G \Lambda^{\frac{1}{2}} \boldsymbol{\xi})_R$, for $\boldsymbol{\xi} \sim \mathcal{N}(0, I)$, gives a sample from the Gaussian random field. Here, $(G \Lambda^{\frac{1}{2}} \boldsymbol{\xi})_R$ corresponds to selecting the first $m_1+1$ row entries in the first $m_2+1$ columns of $G \Lambda^{\frac{1}{2}} \boldsymbol{\xi}$.

\subsection{Smooth Periodisation}
\label{sec: periodisation}

There are instances when the circulant embedding matrix $S$ is not positive definite, depending generally on the covariance function used. As such, the Fourier diagonsalisation cannot be applied directly. To ensure the positive definiteness of $S$, Bachmayr et al. suggest in \citep{bachmayr_unified_2020} periodically extending the covariance function of the Gaussian field and padding the resulting covariance matrix before embedding it to obtain a circulant structure. Here, the analysis is developed for the case when the function $C$ is the Mat\'ern covariance, given by \eqref{eq: matern-cov-function}.
 Without loss of generality, we assume that the sampling domain $D$ is contained in $[0,1]^d$, so that $C : [-1, 1]^d \rightarrow \mathbb{R}$. { Note that for the separable exponential covariance function, given by \eqref{eq: p-norm-cov-function}, the eigenvalues of $S$ are, in fact, real and positive, so that no padding is necessary.}

Specifically, Bachmayr et al. develop two types of periodisation involving a non-smooth and smooth truncation of the covariance function, respectively. The former intrinsically depends on the choice of discretisation parameters $m_i$, $i=1, \dotsc, d$, and length scale $\lambda$, so that the analysis only holds for $\frac{1}{m_i} \ll \lambda$,  $i=1, \dotsc, d$, which is not always satisfied in practice. On the other hand, the smooth periodisation does not involve such an assumption, and, for this reason, this is the method we use in the analysis and numerical examples below. 

Typically, there are two elements required in the periodisation approach: the padding of the covariance matrix and the choice of periodic covariance function $C^{\text{ext}}$. 
The padding technique was first introduced by Graham et al. in \citep{graham_quasi-monte_2011} to ensure the positive definiteness of the embedding matrix $S$. The procedure involves extending the covariance matrix, i.e. ``padding'' it, before mirroring its components. For example, in the two-dimensional case, the extended covariance matrix has the following form:  
\begin{equation*}
    R = \text{Toeplitz}(R_0, R_1, \dotsc, R_{m_2}, R_{m_2+1}, \dotsc, R_{m_2+J_2}),
\end{equation*}
where each $R_j$, $j = 0, \dotsc, m_2+J_2$, is a Toeplitz matrix of dimension $(m_1+J_1+1) \times (m_1+J_1+1)$:
\begin{equation*}
    R_j = \text{Toeplitz}(C_{0, j}, C_{1,j}, \dotsc, C_{m_1, j}, C_{m_1+1, j}, \dotsc, C_{m_1+J_1, j}),
\end{equation*}
where $J_1, J_2 \in \mathbb{N}_0$ denote the padding parameters in the $x$ and $y$ directions, respectively. Upon mirroring, this yields the following structure on the embedding matrix:
\begin{equation*}
    S = \text{Toeplitz}(S_0, S_1, \dotsc, S_{m_2}, S_{m_2+1}, \dotsc, S_{m_2+J_2}, S_{m_2+J_2-1}, \dotsc, S_1),
\end{equation*}
where each block $S_j$ is itself circulant:
\begin{equation*}
    S_j = \text{Toeplitz}(C_{0,j}, C_{1,j}, \dotsc, C_{m_1,j}, C_{m_1+1, j}, \dotsc, C_{m_1+J_1, j}, C_{m_1+J_1-1,j}, \dotsc, C_{1,j}).
\end{equation*}
Note that with slight abuse of notation, we still denote the extended covariance and embedded covariance matrices by $R$ and $S$, respectively. This is to unify and simplify notation. 

Next, the padding values are obtained by evaluating the periodic covariance function $C^{\text{ext}}$ at the points $(i,j)$, with $i = m_1+1, \dotsc, m_1+J_1, \dotsc m_1+1$ and $j = m_2+1, \dotsc, m_2+J_2, \dotsc m_2+1$. 
As explained in \citep{bachmayr_unified_2020}, for the periodisation technique, a $2\ell$-periodic extension of $C$ can be constructed by first choosing a cutoff function $\varphi$ satisfying:
\begin{equation*}
    \varphi = 1 \quad \text{on} \quad [-1,1] \quad \text{and} \quad \varphi=0 \quad \text{on} \quad \mathbb{R} \setminus [-\kappa, \kappa], \quad \kappa \coloneqq 2\ell -1 \geq \ell,
\end{equation*}
where $\text{supp}(C)=[-1,1]^d$ and $\ell = 1+J/m \geq 1$ with $J$ padding and $m$ discretisation parameters. Here, we assume, for simplicity, that $J \equiv J_i$ and $m \equiv m_i$, for $i = 1, \dotsc, d$.
Then, the periodic extension $C^{\text{ext}}$ is given by:
\begin{equation}
\label{eq: periodic-cov-fun}
    C^{\text{ext}}(\mathbf{x}) = \sum_{\mathbf{n} \in \mathbb{Z}^d} (C \varphi_\kappa) (\mathbf{x} + 2 \ell \mathbf{n}), \quad \mathbf{x} \in \mathbb{R}^d,
\end{equation}
where $\varphi_\kappa(\mathbf{x}) \coloneqq \varphi(\|\mathbf{x}\|_\infty)$. For the smooth periodisation case, $\varphi$ must also satisfy the following:
\begin{equation*}
    \ell > 1 \quad \text{and} \quad \varphi \in C_0^\infty(\mathbb{R}), \quad \text{with} \quad \text{supp}(\varphi) = [-\kappa, \kappa].
\end{equation*}
An example of such a cutoff function which is easy to implement in practice is given by:
\begin{equation*}
    \varphi(t) = \frac{\eta\left(\frac{\kappa-|t|}{\kappa-1}\right)}{\eta\left(\frac{\kappa-|t|}{\kappa-1}\right) + \eta\left(\frac{|t|-1}{\kappa-1}\right)},
\end{equation*}
where:
\begin{equation*}
    \eta(x) = \begin{cases}
        \exp(-x^{-1}), \quad &x > 0, \\
        0, &x \leq 0.
    \end{cases}
\end{equation*}
In addition, to ensure that $C^{\text{ext}} = C$ on $[-1,1]^d$ we require \citep{bachmayr_unified_2020}:
\begin{equation*}
    \ell \geq \frac{\kappa +  \sqrt{d}}{2}.    
\end{equation*}
Further, Theorem 10 in \citep{bachmayr_unified_2020} gives the following lower bound on $\kappa$ such that the embedding matrix is symmetric positive definite:
\begin{equation}
\label{eq: padding-bound}
    \frac{\kappa}{\lambda} \geq C_1 + C_2 \max \left\{\nu^{\frac{1}{2}}(1 + |\ln(\nu)|), \nu^{-\frac{1}{2}}\right\}.
\end{equation}

\subsection{Eigenvalue Decay}
\label{sec: eigenvalues-decay}

{Throughout the remainder of the paper, let $\{\Lambda^{\text{ord}}_j\}_{j=1}^s$ be the circulant embedding eigenvalues $\{\Lambda_j\}_{j=1}^s$ arranged in non-increasing order.} { The main results in this section, given in Lemma \ref{thm: matern-egnv-decay} and \ref{lemma: expo-egnv-decay} below, provide a decay rate (in $j$) of $\{\Lambda^{\text{ord}}_j\}_{j=1}^s$. This will be a key factor in the study of the proposed smoothing technique.}
{ The proofs of Lemma \ref{thm: matern-egnv-decay} and \ref{lemma: expo-egnv-decay} are based on the following two main ideas. 

Firstly, as shown in Theorem 3.1 in \cite{graham_analysis_2018}, the scaled (discrete) eigenvalues of the extended circulant embedding matrix of the periodised covariance $C^{\text{ext}}$ converge to the (continuous) eigenvalues of the covariance operator of $C^{\text{ext}}$ as the mesh sizes $h_i, i=1,\dots,d$, tend to zero.
This can be seen by comparing the expression for the discrete eigenvalues 
$\{\Lambda_j\}_{j=1}^s$, obtained from the discrete Fourier transform as described in section \ref{sec: ce}, with that of the continuous eigenvalues, given by the Fourier transform of $C^{\text{ext}}$:
\begin{equation*}
    \lambda_\mathbf{j}^{\text{ext}} = \int_{[-\ell, \ell]^d} C^{\text{ext}}(\mathbf{x}) \exp\left(\frac{\pi \mathrm{i} \mathbf{j} \mathbf{x}}{\ell} \right) \mathrm{d}\mathbf{x}, \quad \mathbf{j} \in \mathbf{Z}^d.
\end{equation*}
The discrete eigenvalues, when scaled by $\prod_{i=1}^d m_i^{-1}$, can thus be seen as a rectangle rule approximation of the continuous eigenvalues. 

Secondly, it is shown in \citep{bachmayr_representations_2018} that the ordered eigenvalues $\{\lambda_j^{{\text{ext}}}\}_{j=1}^s$ of the periodised covariance $C^{\text{ext}}$ in the smooth periodisation case decay at precisely the same rate as the eigenvalues $\{\lambda_j\}_{j=1}^s$ of the covariance operator of the original covariance $C$. This is proven for a general class of covariance functions in \citep[Eq. (64) and (65)]{bachmayr_representations_2018}, however here we focus on the Mat\'ern case, stated in Lemma \ref{thm: matern-egnv-decay} below. More details on this are given in Remark \ref{rem:KL_CE}.
}

{Finally, the following two lemmas give the decay rate of the ordered eigenvalues $\{\Lambda_j^{\text{ord}}\}_{j=1}^s$ for the case of Mat\'ern and separable exponential covariance functions, respectively.} { Following from the observations above, the decay rate we obtain for the discrete eigenvalues $\{\Lambda^{\text{ord}}_j\}_{j=1}^s$ is the same as that of the continuous eigenvalues $\{\lambda_j\}_{j=1}^\infty$ of the covariance operator of the original covariance $C$.}

\begin{Lemma}
\label{thm: matern-egnv-decay}
    Let $C$ be the Mat\`ern covariance  \eqref{eq: matern-cov-function}. {Suppose the padding parameters $J_i,  i = 1, \dotsc, d$, are chosen such that $\kappa$ satisfies Eq. \eqref{eq: padding-bound}}. Let $\{\Lambda_j^{{\text{ord}}}\}_{j=1}^s$ denote the {ordered} eigenvalues of the (extended) embedding matrix $S \in \mathbb{R}^{s \times s}$, with $s= 2^d \prod_{i=1}^d (m_i + J_i)$. Then for $j = 1, \dots, s$:
    \begin{equation}
    \label{eq: egnv-matern-cov-decay}
        \left(\prod_{i=1}^d m_i^{-1} \right) \, \Lambda_j^{{\mathrm{ord}}} \lesssim j^{-1-\frac{2\nu}{d}}.
    \end{equation}
\end{Lemma}
\begin{proof}
{To find the decay rate of the discrete eigenvalues, we note that:
\begin{equation*}
    \left(\prod_{i=1}^d m_i^{-1} \right) \Lambda_j = \left|\left(\prod_{i=1}^d m_i^{-1} \right) \Lambda_j - \lambda^{\text{ext}}_j + \lambda^{\text{ext}}_j \right| \leq \left|\left(\prod_{i=1}^d m_i^{-1} \right) \Lambda_j - \lambda^{\text{ext}}_j \right| + |\lambda^{\text{ext}}_j|.
\end{equation*}
Then, because the smoothed periodisation preserves the regularity of the Mat\'ern covariance function by construction, we can use Lemma 1 in \cite{bachmayr_unified_2020} 
to infer that the discrete eigenvalues converge to the continuous eigenvalues faster than $|\lambda^{\text{ext}}_j|$ decays. Hence, $\left\{\left(\prod_{i=1}^d m_i^{-1} \right) \Lambda^{\text{ord}}_j\right\}_{j=1}^s$ decay at the same rate as $\lambda_j^{\text{ext}}$, and, therefore, as $\lambda_j$ {by \citep[Eq. (64) and (65)]{bachmayr_representations_2018}}.} The decay rate for the continuous eigenvalues is given in \cite[Corollary 5]{graham_quasi-monte_2015}, which completes the proof.
\end{proof}

\begin{Lemma}
\label{lemma: expo-egnv-decay}
    Let $C$ be the separable exponential covariance  \eqref{eq: p-norm-cov-function}, and let $\{\Lambda_j^{{\text{ord}}}\}_{j=1}^s$ denote the {ordered} eigenvalues of the embedding matrix $S \in \mathbb{R}^{s \times s}$, with $s=2^d \prod_{i=1}^d m_i$. Then for $j = 1, \dots, s$:
    \begin{equation}
    \label{eq: egnv-real-and-positive}
        \Lambda_j \in \mathbb{R} \quad \text{ and } \quad \Lambda_j > 0,
    \end{equation}
     and
    \begin{equation}
    \label{eq: egnv-exp-cov-decay}
        \left(\prod_{i=1}^d m_i^{-1} \right) \, \Lambda_j^{{\mathrm{ord}}}  \lesssim {j^{-2} (\log j)^{2(d-1)}, \quad \forall d \geq 1.}
    \end{equation}
\end{Lemma}
\begin{proof}
    First, Eq. \eqref{eq: egnv-real-and-positive} is a direct application of Theorem 2 in \citep{dietrich_fast_1997}. This implies that no padding is necessary, as the embedding matrix is symmetric positive definite.

    Second, for Eq. \eqref{eq: egnv-exp-cov-decay}, note that in the one-dimensional case, the exponential covariance is equivalent with the Mat\'ern covariance with $\nu = 0.5$. Hence, we can apply Eq. \eqref{eq: egnv-matern-cov-decay} with no periodisation to infer that the corresponding eigenvalues satisfy:
        \begin{equation*}
            \left(\prod_{i=1}^d m_i^{-1} \right) \, \Lambda_j^{{\mathrm{ord}}}  \lesssim j^{-2}.
        \end{equation*}

    To extend this to higher dimensions, we exploit the fact that the exponential covariance with norm $p=1$ is separable, so that:
    \begin{equation*}
        \mathbf{s}^{dD} = \frac{1}{\sigma^n} \mathbf{s}^{1D} \otimes \dotsc \otimes \mathbf{s}^{1D},
    \end{equation*}
    where $\mathbf{s}^{dD}$ is the first {column} of the $d$-dimensional embedding matrix $S$, and $\otimes$ denotes the Kronecker product. Further, the discrete Fourier matrix satisfies \citep{cooley_finite_1969}:
    \begin{equation*}
        F^{ndD} = F^{1D} \otimes \dotsc \otimes F^{1D},
    \end{equation*}
    so that:
    \begin{equation}
    \label{eq: n-dim-egnv}
        \boldsymbol{\Lambda}^{dD} = \boldsymbol{\Lambda}^{1D} \otimes \dotsc \otimes \boldsymbol{\Lambda}^{1D}.
    \end{equation}
    {We then use Theorem 1 in \cite{krieg_tensor_2018} together with the decay rate of the one-dimensional eigenvalues above to infer that:
    \begin{equation*}
        \left(\prod_{i=1}^d m_i^{-1} \right) \Lambda^{dD, {{\mathrm{ord}}}}_j \lesssim j^{-2} (\log j)^{2(d-1)},
    \end{equation*}
    with a constant that depends on the dimension $d$.
    }
This completes the proof.
\end{proof}

{
\begin{Remark}\label{rem:KL_CE} As the results in this section show, there is an intricate link between the circulant embedding method and the KL expansion: the discrete and continuous eigenvalues decay at the same rate, and the (scaled) discrete eigenvalues converge to the continuous eigenvalues. Furthermore, consider the KL expansion of the continuous random field $Z$, given by:
\begin{equation}
\label{eq: kl-expansion}
    Z(\mathbf{x}, \omega) = \sum_{j=1}^\infty \sqrt{\lambda}_j b_j(\mathbf{x}) \xi_j(\omega),
\end{equation}
where $\{\xi_j\}_{j \in \mathbb{N}} \sim \mathcal{N}(0,1)$ i.i.d., and $\{\lambda_j\}_{j \in \mathbb{N}}$ and $\{b_j\}_{j \in \mathbb{N}}$ are the eigenvalues and $L^2(D)$-orthonormal eigenfunctions of the covariance operator, respectively.
{Bachmayr et al. show in \citep[Section 3]{bachmayr_representations_2018} that, when extending the domain $D \subset \mathbb{R}^d$ of a random field to a periodic domain $\mathbb{T}^d$, the eigenfunctions of the corresponding periodized covariance are explicitly known trigonometric functions (see \citep[Eq. (57)]{bachmayr_representations_2018}), and the eigenvalues are Fourier coefficients of an even and $\mathbb{T}$-periodic function (see \citep[Eq. (59)]{bachmayr_representations_2018}). By restricting back to the original domain $D$, we obtain an alternative series (KL-type) expansion of the random field on the original domain $D$. Evaluated on a grid, the expressions for the discrete random field $\mathbf{Z}$ obtained from the circulant embedding method and this KL-type expansion coincide due to aliasing effects. Note however, that the smoothing technique introduced in this work and the truncation of the KL-type expansion (similar to \cite{teckentrup_further_2013}) are not equivalent.}
\end{Remark}
}

\subsection{Smoothing}
\label{sec: smoothing}

Here, we focus on the case when the random field $Z$ we aim to sample from is extremely oscillatory, so that its fluctuations can only be resolved on a fine mesh, which can render associated simulations computationally expensive. To this end, we smooth the generated realisation by setting the smallest eigenvalues in the factorisation of the circulant embedding matrix, which correspond to highly oscillatory terms, to zero. This approach is similar in nature to the level-dependent truncation of the KL-expansion proposed in \citep{teckentrup_further_2013}.

With this in mind, recall that given a zero-mean stationary Gaussian random field $Z$, we now have $\mathbf{Z} = (G \Lambda^{\frac{1}{2}} \boldsymbol{\xi})_R$, for $\boldsymbol{\xi} \sim \mathcal{N}(0, I)$, a sample from the discrete representation of $Z$. Here, $\Lambda$ is a diagonal matrix whose entries are the eigenvalues $\{\Lambda_j\}_{j=1}^{s}$ corresponding to the embedding matrix $S \in \mathbb{R}^{s \times s}$.  {We define the  truncated eigenvalues $\{\tilde{\Lambda}_j^{\text{ord}}\}_{j=1}^s$ as follows:
\begin{equation}
\label{eq: tau-definition}
    \tilde{\Lambda}_j^\text{ord} = 
    \begin{cases}
        \Lambda_j^\text{ord} &\text{ if } j \leq \tau, \\
        0 &\text{ if } j > \tau,
    \end{cases}
\end{equation}
such that $\tau$ denotes the number of smallest eigenvalues we discard.} Then $\tilde{\Lambda}$ is a diagonal matrix whose entries are the truncated eigenvalues $\{\Lambda^{{\mathrm{ord}}}_j\}_{j=1}^{s-{\tau}}$ and $\mathbf{0} \in \mathbb{R}^{{\tau}}$. Further, let $\tilde{\mathbf{Z}} = (G \tilde{\Lambda}^{\frac{1}{2}} \boldsymbol{\xi})_R$ represent the corresponding smoothed realisation. 

Then, the underlying smoothing procedure is as follows:
\begin{enumerate}
	\item sort eigenvalues $\{\Lambda_j\}_{j=1}^s$ in non-increasing order {to obtain $\{\Lambda_j^\mathrm{ord}\}_{j=1}^s$};
	\item choose truncation index ${\tau}$ such that $\mathbb{E} \left[\|\mathbf{Z} - \tilde{\mathbf{Z}}\| \right]$ is less than a given threshold $\varepsilon$;
	\item compute smoothed realisation $\tilde{\mathbf{Z}}$ using the same $\boldsymbol{\xi}$ as in the definition of $\mathbf{Z}$.
\end{enumerate}
The choice of norm for the smoothing error $\mathbb{E} \left[\|\mathbf{Z} - \tilde{\mathbf{Z}}\| \right]$ is application dependent. Throughout the paper, we use the $L^\infty$ norm, as this is the natural norm that emerges in the finite element discretisation error analysis for the application at hand, as shown in Theorem \ref{thm: error-functional-sample} and Corollary \ref{cor: error-functional-egnv}. For different applications, one can adapt these results using the equivalence of norms. {However, this will lead to different convergence rates due to the presence of {\color{red}s}-dependent factors.}

For example, Figure \ref{fig: smooth-and-non-smooth-comp} depicts a Gaussian field sample with covariance function \eqref{eq: p-norm-cov-function}, where we take $D=[0,1]^2$ {discretised using $m_1 = m_2 = 32$}, norm $p=1$ and parameters $\sigma^2=1$ and $\lambda = 0.1$. In addition, in Figures \ref{fig: surface-plot-smoothed} and \ref{fig: contour-plot-smoothed} we drop ${\tau} = \frac{7s}{8}$ eigenvalues, yielding a noticeably smoother approximation of the sample in Figures \ref{fig: surface-plot-not-smoothed} and \ref{fig: contour-plot-not-smoothed}.

\begin{figure*}[t]
        \centering
        \begin{subfigure}[t]{0.49\textwidth}
            \centering
            \includegraphics[width=1\textwidth]{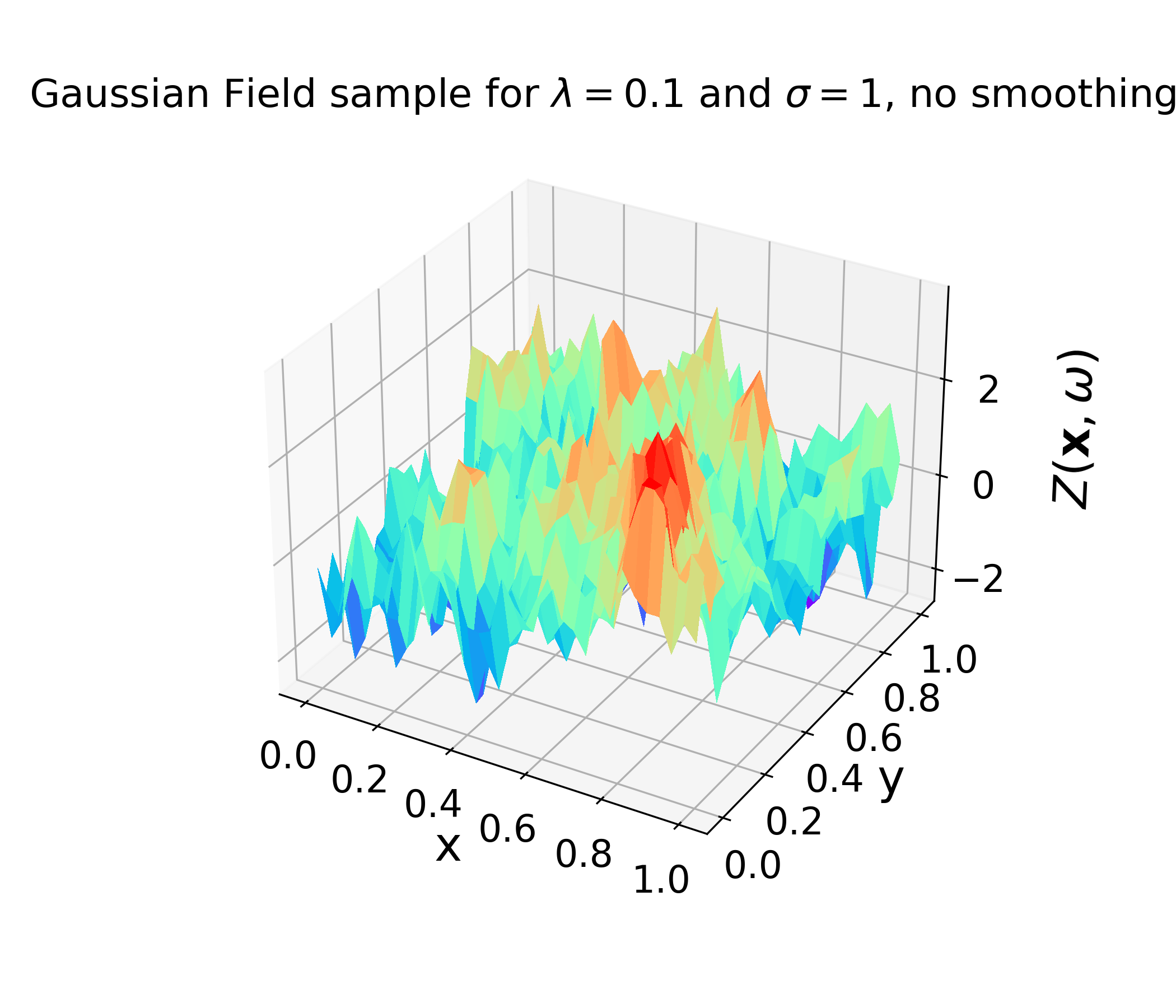}
            \caption{Surface plot of $\mathbf{Z}(\mathbf{x}, \omega)$.}    
            \label{fig: surface-plot-not-smoothed}
        \end{subfigure}
        \hfill
        \begin{subfigure}[t]{0.49\textwidth}  
            \centering 
            \includegraphics[width=\textwidth]{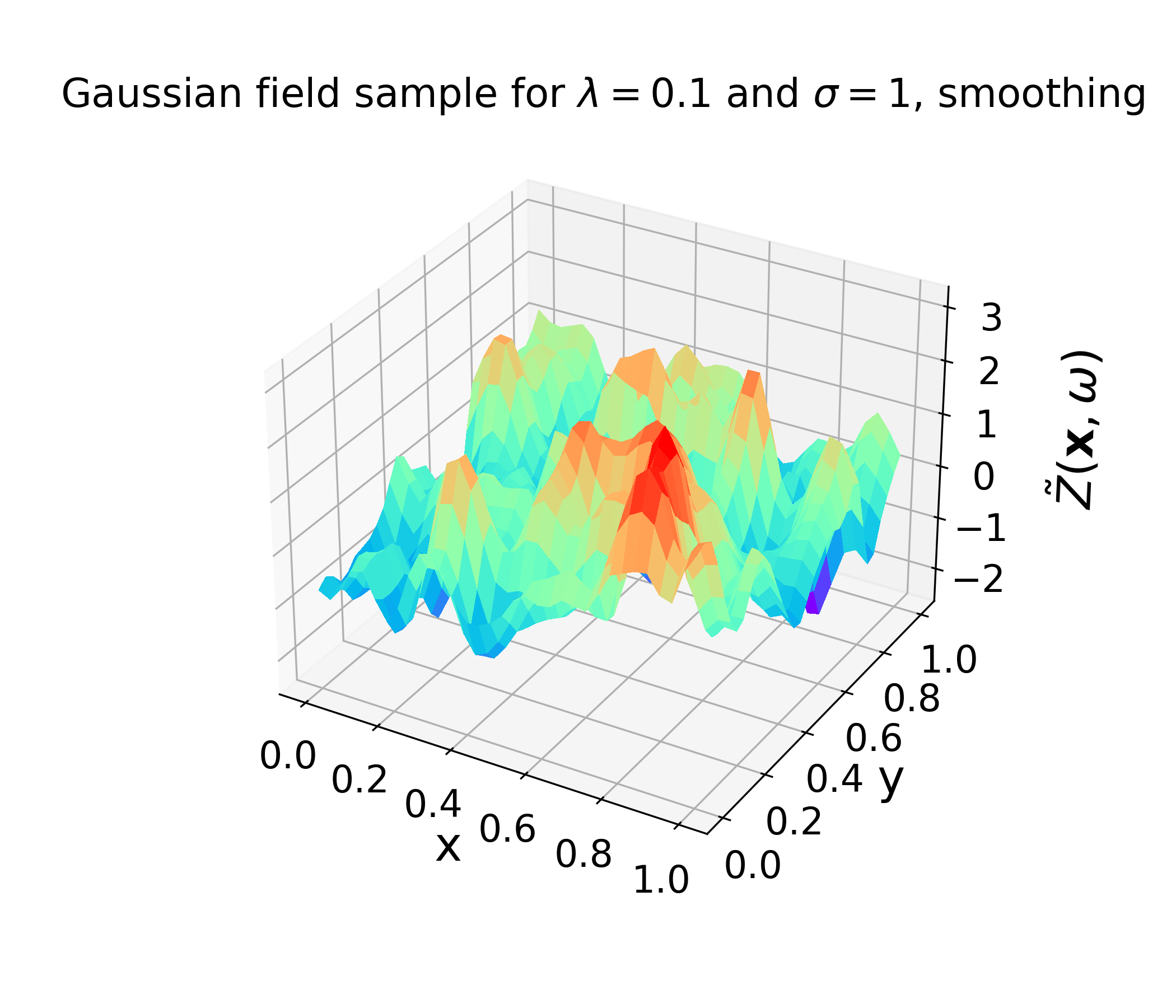}
            \caption{Surface plot of $\tilde{\mathbf{Z}}(\mathbf{x}, \omega)$.}    
            \label{fig: surface-plot-smoothed}
        \end{subfigure}
        \vskip\baselineskip
        \begin{subfigure}[t]{0.5\textwidth}   
            \centering 
            \includegraphics[width=\textwidth]{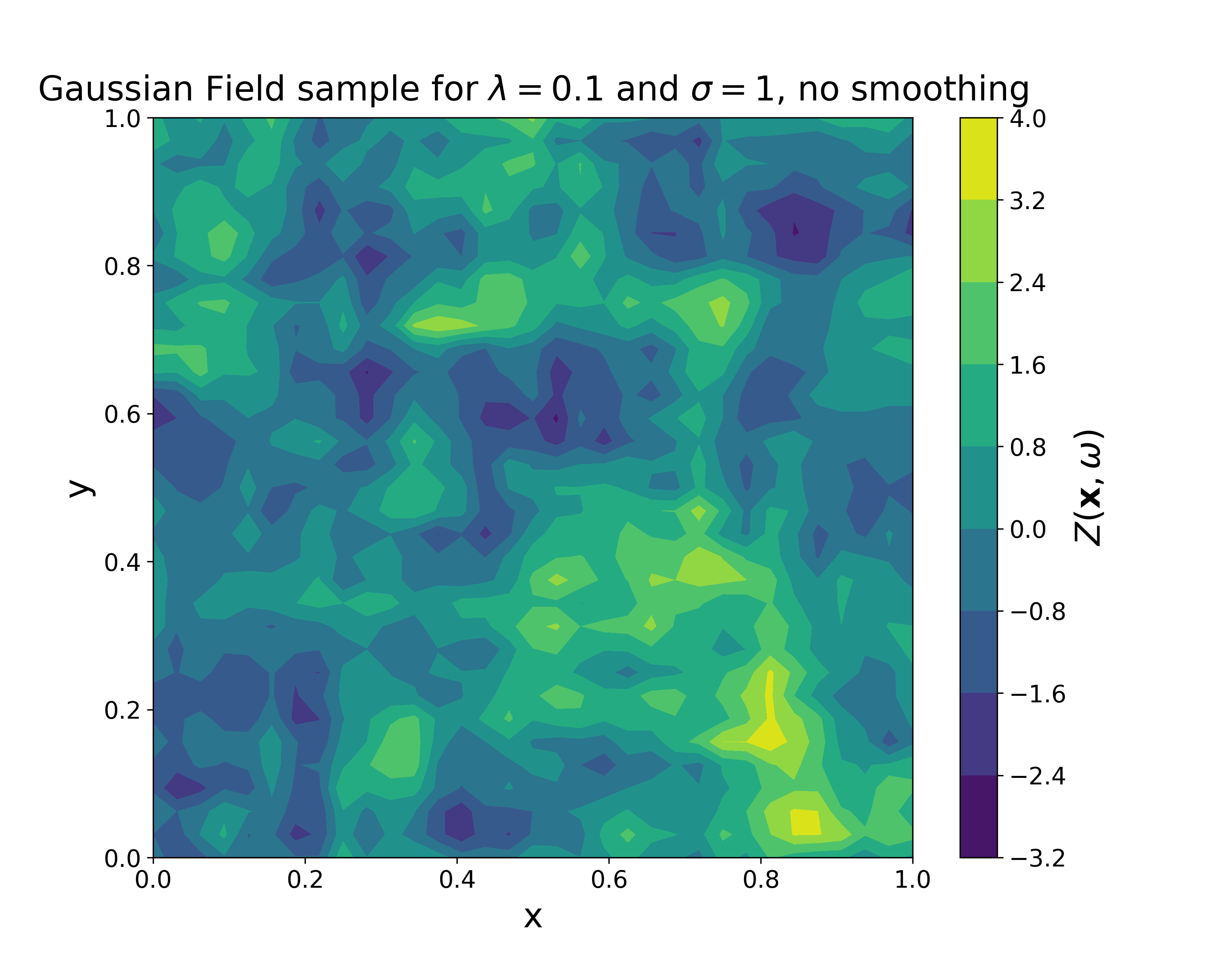}
            \caption{Contour plot of $\mathbf{Z}(\mathbf{x}, \omega)$.}    
            \label{fig: contour-plot-not-smoothed}
        \end{subfigure}
        \hfill
        \begin{subfigure}[t]{0.49\textwidth}   
            \centering 
            \includegraphics[width=\textwidth]{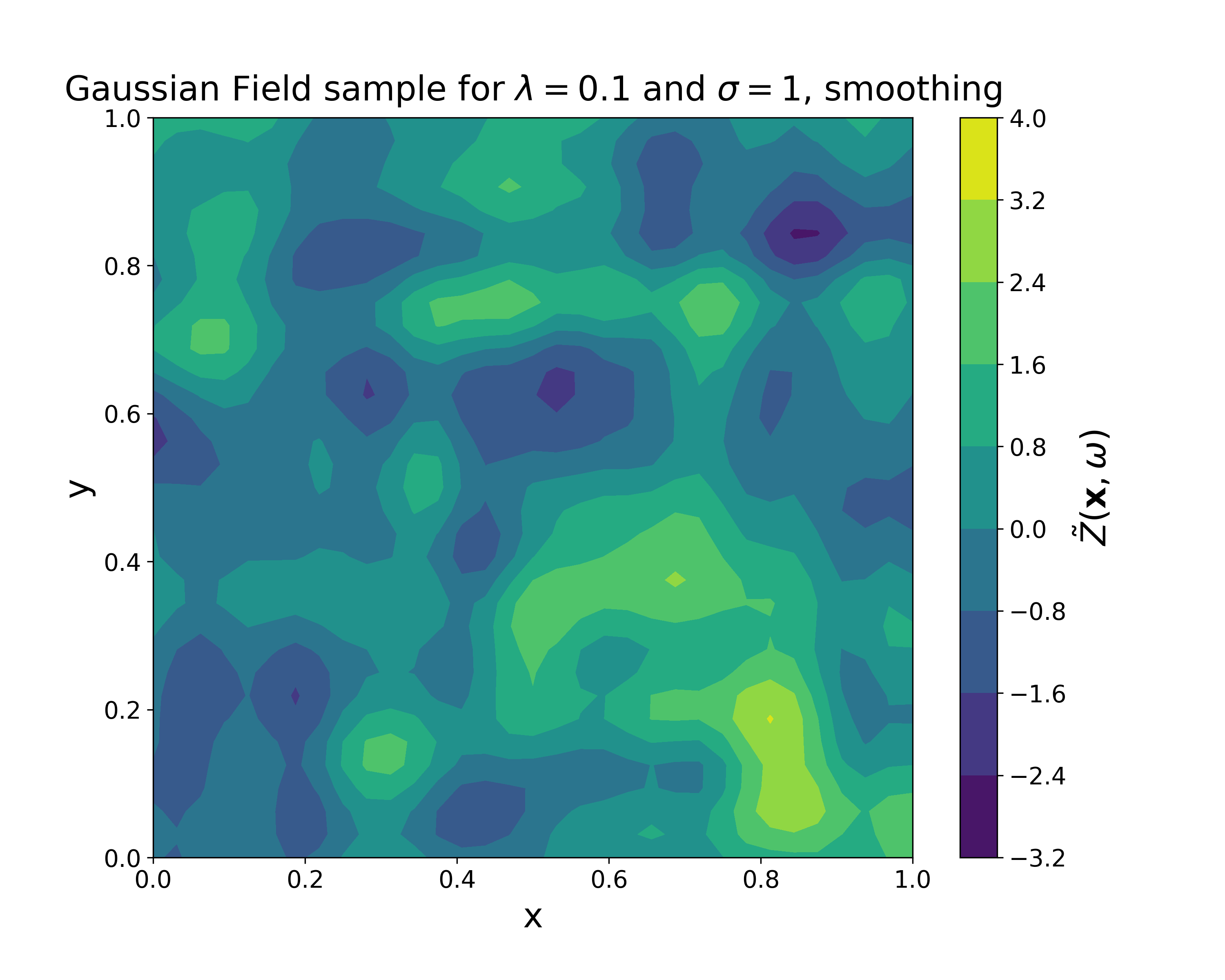}
            \caption{Contour plot of $\tilde{\mathbf{Z}}(\mathbf{x}, \omega)$.}
            \label{fig: contour-plot-smoothed}
        \end{subfigure}
        \caption{Plots of Gaussian field samples $\mathbf{Z}(\mathbf{x}, \omega)$ and $\tilde{\mathbf{Z}}(\mathbf{x}, \omega)$ obtained using circulant embedding for the exponential covariance \eqref{eq: p-norm-cov-function} with {$m_1 = m_2 = 32$}, $p=1$, $\sigma^2=1$ and $\lambda=0.1$.} 
        \label{fig: smooth-and-non-smooth-comp}
\end{figure*}

To be able to implement this algorithm in practice, we require a theoretical bound on the error between the original and the smoothed samples $\mathbf{Z}$ and $\tilde{\mathbf{Z}}$, respectively. This is given in the following theorem:
\begin{Theorem}
\label{thm: sample-error}
Let $Z$ be a $d$-dimensional zero-mean stationary Gaussian random field, and let $\mathbf{Z}$ be the discrete representation of $Z$ obtained using the circulant embedding method on a uniform grid $\mathcal{T}$ with mesh sizes $m_i$, $i=1, \dotsc, d$, of the domain $D \subset \mathbb{R}^d$. Suppose the embedding matrix $S \in \mathbb{R}^{s \times {s}}$ is symmetric positive definite upon smooth periodisation with padding parameters $\{J_i\}_{i=1}^d \subseteq \mathbb{N}_0^d$. Let {$\tau$} be {as in Eq. \eqref{eq: tau-definition}} and $\tilde{\mathbf{Z}}$ be the resulting smoothed sample. Then, for any $p \in [1, \infty)$:
\begin{equation*}
	\mathbb{E}\left[\|\mathbf{Z} - \tilde{\mathbf{Z}}\|^p_\infty \right] \lesssim s^{-\frac{p}{2}} \left(\Lambda_{{s-{\tau}+1}}^{{\mathrm{ord}}}\right)^\frac{p}{2} \, {\tau}^p,
\end{equation*}
where $s = 2^d \prod_{i=1}^d (m_i + J_i)$.
\end{Theorem}
\begin{proof}
Using the definition of $\mathbf{Z}$ and $\tilde{\mathbf{Z}}$, note that:
\begin{align*}
    \mathbb{E}\left[\, \|\mathbf{Z} - \tilde{\mathbf{Z}}\|^p_\infty\right] &= \mathbb{E}\left[\left\|\left(G\Lambda^{\frac{1}{2}} \boldsymbol{\xi} \right)_R - \left(G \Tilde{\Lambda}^{\frac{1}{2}} \boldsymbol{\xi} \right)_R\right\|^p_\infty\right]
    \nonumber \\
    &\lesssim \mathbb{E}\left[\left\|G \left(\Lambda^{\frac{1}{2}} - \Tilde{\Lambda}^{\frac{1}{2}}\right) \boldsymbol{\xi} \right\|^p_\infty\right]
    \nonumber \\ 
    &= \mathbb{E}\left[\max_{i=1, \dotsc, s} \left|\sum_{j=s-{\tau}+1}^s g_{ij} \sqrt{\Lambda_j^{{\mathrm{ord}}}} \xi_j \right|^p \,\right] \nonumber \\
    &\lesssim \mathbb{E}\left[\max_{i=1, \dotsc, s} \left(\max_{j=s-{\tau}+1, \dotsc s} |g_{ij}| \, \sqrt{\Lambda_{{s-{\tau}+1}}^{{\mathrm{ord}}}} \sum_{j=s-{\tau}+1}^s |\xi_j| \right)^p \,\right] \nonumber \\
    &= \max_{i=1, \dotsc, s} \max_{j=s-{\tau}+1, \dotsc s} |g_{ij}|^p \left(\Lambda_{{s-{\tau}+1}}^{{\mathrm{ord}}}\right)^\frac{p}{2} \mathbb{E}\left[\left( \sum_{j=s-{\tau}+1}^s |\xi_j| \right)^p \,\right].
\end{align*}

Now, $g_{ij} = \Re(f_{ij}) + \Im(f_{ij})$, where $f_{ij}$ is an entry in the $d$-dimensional unitary Discrete Fourier Transform (DFT). This is a tensor product of $d$ one-dimensional unitary DFT's \citep{cooley_finite_1969}, whose entries are given in Eq. \eqref{eq: Fourier-matrix-entries}, so that the following holds:
\begin{equation}
\label{eq: fourier-entry-bound}
    |f_{ij}| \leq \frac{1}{\sqrt{s}}.
\end{equation}

Further, applying the identity $|\Re(z) + \Im(z)| \leq \sqrt{2} |z|$, for any $z \in \mathbb{C}$, to $g_{ij} = \Re(f_{ij}) + \Im(f_{ij})$, together with Eq. \eqref{eq: fourier-entry-bound} we derive the following:
\begin{equation*}
     \max_{i=1, \dotsc, s} \max_{j=s-{\tau}+1, \dotsc s} |g_{ij}| \leq \sqrt{\frac{2}{s}},
\end{equation*}
which yields:
\begin{equation*}
    \mathbb{E}\left[\, \|\mathbf{Z} - \tilde{\mathbf{Z}}\|^p_\infty\right] \lesssim s^{-\frac{p}{2}} \left(\Lambda_{{s-{\tau}+1}}^{{\mathrm{ord}}}\right)^\frac{p}{2} \mathbb{E}\left[\left( \sum_{j=s-{\tau}+1}^s |\xi_j| \right)^p \,\right].
\end{equation*}
Next note that, using the linearity of expectation, we obtain the following:
\begin{equation*}
     \mathbb{E}\left[\left( \sum_{j=s-{\tau}+1}^s |\xi_j| \right)^p \,\right] = \sum_{i_1=s-{\tau}+1}^s \dotsc \sum_{i_p=s-{\tau}+1}^s \mathbb{E}\left[\prod_{m=1}^p |\xi_{i_m}|\right].
\end{equation*}
Since some elements in the sum above are correlated, we observe that every term is, in fact, an element of the following set:
\begin{equation*}
    \Xi^p = \left\{\prod_{i=1}^m \mathbb{E}[|\xi^{p_i}|] : p_i \in \{1, 2, \dotsc, p\}, i=1, \dotsc, m, \sum_{i=1}^m p_i = p \right\},
\end{equation*}
{and we denote them by $q_i$}, where $\xi \sim \mathcal{N}(0,1)$. Hence, we can write:
\begin{align}
\label{proof: thm-1-ineq}
     \mathbb{E}\left[\left( \sum_{j=s-{\tau}+1}^s |\xi_j| \right)^p \,\right] &\leq {\tau}^p \max_{m=1, \dotsc, p} \, \, \max_{i_m = s-{\tau}+1, \dots, s} \mathbb{E}\left[\prod_{m=1}^p |\xi_{i_m}|\right] \nonumber \\
     &= {\tau}^p \max_{q_i \in \Xi^p} q_i \nonumber \\
     &= {\tau}^p q^*. 
\end{align}
In the second step, we used that $\mathbb{E}\left[\prod_{m=1}^p |\xi_{i_m}|\right] \in \Xi^p$ by definition of the set $\Xi^p$. Since $\Xi^p$ is finite, its maximum element, which we denote by $q^*$, will also be an element of $\Xi^p$. {Then, let $p_i^*$ and $m^*$ be the exponents and their number, respectively, for which:
$$
    q^* = \prod_{i=1}^{m^*} \mathbb{E}[|\xi|^{p_i^*}],
$$}
with $\sum_{i=1}^{m^*} p_i^* = p$ and $p_i^* \in \{1, 2, \dotsc, p\}, i=1, \dotsc, m^*$. {Using this, the inequality in \eqref{proof: thm-1-ineq} becomes:}
$$
    \mathbb{E}\left[\left( \sum_{j=s-{\tau}+1}^s |\xi_j| \right)^p \,\right] \leq {\tau}^p \prod_{i=1}^{m^*} \mathbb{E}[|\xi|^{p_i^*}]
$$
Further, using Jensen's inequality for $\varphi(x) = |x|^{\frac{p}{p^{{*}}_i}}$, we obtain:
\begin{equation*}
    \mathbb{E}[|\xi|^{p^{{*}}_i}] \leq \mathbb{E}[|\xi|^p]^{\frac{p^{{*}}_i}{p}},
\end{equation*}
for $i=1, \dotsc, m^{{*}}$, so that:
\begin{equation*}
    \prod_{i=1}^m \mathbb{E}[|\xi|^{p_i^{{*}}}] \leq \mathbb{E}[|\xi|^p]^{\frac{\sum_{i=1}^{m^{{*}}} p_i^{{*}}}{p}} = \mathbb{E}[|\xi|^p].
\end{equation*}
Putting everything together, we determine:
\begin{equation*}
    \mathbb{E}\left[\left( \sum_{j=s-{\tau}+1}^s |\xi_j| \right)^p \,\right] \leq {\tau}^p \, \, \mathbb{E}[|\xi|^p], \quad p \in \mathbb{N}.
\end{equation*}

From \citep{elandt_folded_1961}, we know that if $\xi \sim \mathcal{N}(\mu, \sigma^2)$, then $|\xi|$ is a folded normal random variable. For the special case when $\mu = 0$, $|\xi|$ is a half-normal random variable, whose higher moments are:
\begin{align*}
    \mathbb{E}[|\xi|^{2n}] &= \frac{(2n)!}{n! \, 2^n}, \nonumber \\
    \mathbb{E}[|\xi|^{2n+1}] &= \sqrt{\frac{2}{\pi}} 2^n n! \, ,
\end{align*}
for $n \in \mathbb{N} \cup \{0\}$ and $\sigma =1$.
Using these formulae, we can infer the following:
\begin{equation*}
     \mathbb{E}\left[\left( \sum_{j=s-{\tau}+1}^s |\xi_j| \right)^p \,\right] \leq 
     \begin{cases}
         & \frac{(2n)!}{n! \, 2^n} \, \, k^{2n}, \qquad  \quad \text{ if } p = 2n, \\
         & \sqrt{\frac{2}{\pi}} 2^n n! \, \, k^{2n+1}, \, \, \text{ if } p = 2n +1. 
     \end{cases}
\end{equation*}
Re-arranging, this reduces to:
\begin{equation*}
    \mathbb{E}\left[\left( \sum_{j=s-{\tau}+1}^s |\xi_j| \right)^p \,\right] \lesssim {\tau}^p,
\end{equation*}
with constant dependent on $p$, but not on ${\tau}$. 

Combining everything gives us the claimed result for $p \in \mathbb{N}$, but this can be extended to $p \in [1, \infty)$ using Jensen's inequality. The claim of the Theorem then follows.
\end{proof}

In practice, we observe that the error $\mathbb{E}\left[\|\mathbf{Z} - \tilde{\mathbf{Z}}\|_\infty\right]$ decays much faster {for small values of ${\tau}$}. This is to be expected, as the case ${\tau} = 0$ yields $\mathbf{Z} - \tilde{\mathbf{Z}} = \mathbf{0}$. As the number of eigenvalues we discard increases, we recover the decay rate that we theoretically derive. What is more, this bound is directly linked with the largest eigenvalue we keep, so that the faster the eigenvalues of the embedding matrix decline, the faster the smoothing error decreases.

\section{Multilevel Monte Carlo Methods}
\label{sec: mlmc}

In this section, we focus on an individual application of the smoothing technique, namely Multilevel Monte Carlo methods for PDEs with random coefficients. We first introduce the PDE model of interest and outline key features of the MC and MLMC estimators, where we make use of circulant embedding methods for sampling from the aforementioned random parameter. Finally, we describe how to integrate the smoothing approach in these methods, and give estimates on the error introduced by smoothing samples from the random field and the complexity of the resulting estimators. 

\subsection{PDE with random coefficient model}
\label{sec: pde-model}

A mathematical model for simulating single-phase, stationary groundwater flow rests on Darcy's Law combined with the \textit{Law of conservation of mass} \citep{de_marsily_quantitative_1986}, which leads to a linear relationship between the pressure difference of the groundwater and its flow rate, as follows:
\begin{equation}
\label{eq: pde-model}
    - \nabla \cdot (k(\mathbf{x}, \omega) \nabla u(\mathbf{x}, \omega)) = f(\mathbf{x}), \quad \mathbf{x} \in D.
\end{equation}
For simplicity, we let $D =  (0,1)^2$, and take the boundary conditions to be:
\begin{align*}
    u\rvert_{x_1=0} &= 1, \qquad u\rvert_{x_1=1}=0, \nonumber \\
    \frac{\partial u}{\partial \mathbf{n}} \biggr\rvert_{x_2=0} &= 0, \quad \frac{\partial u}{\partial \mathbf{n}} \biggr\rvert_{x_2=1} = 0.
\end{align*}
Here, $u$ denotes the pressure of the fluid, and $f \coloneqq - \nabla \cdot \mathbf{g}$, where $\mathbf{g}$ represents the source terms. Further, $k$ stands for the hydraulic conductivity, that is, the ease with which a fluid can move through porous media or fractures under a given pressure gradient. Moreover, we apply the flow cell boundary conditions, which comprise of Dirichlet boundary conditions, ensuring that the groundwater stream is from $x_1 = 0$ to $x_1 = 1$, and Neumann boundary conditions, forcing the flow to remain vertically inside the domain.

To account for uncertainties in the coefficient $k$, we assume that it is a log-normal random field, i.e. $k(\mathbf{x}, \omega) = \exp(Z(\mathbf{x}, \omega))$. Here, $Z$ is a zero-mean stationary Gaussian field, so that it has the structure outlined in Eq. \eqref{eq: rf-mean-cov}. 
One example of covariance function $C : [-1,1]^2 \rightarrow \mathbb{R}$ associated with the coefficient $k$ was suggested by Hoeksema and Kitanidis in \citep{hoeksema_analysis_1985}, and it has the form given in Eq. \eqref{eq: p-norm-cov-function}. In this instance, samples $Z(\cdot, \omega)$ are H\"{o}lder continuous with respect to $\mathbf{x}$, with H\"{o}lder exponent $\alpha < 1/2$, so that:
\begin{equation*}
    |Z(\mathbf{x}, \cdot) - Z(\mathbf{y}, \cdot)| \lesssim \|\mathbf{x} - \mathbf{y}\|^\alpha, \quad \forall \, \mathbf{x}, \mathbf{y} \in D.
\end{equation*}
Another example of $C$ commonly used in this application is the Mat\'ern covariance function given in Eq. \eqref{eq: matern-cov-function}. 

Generally, we are interested in computing quantities of interest $\mathbb{E}[Q]$ in the form of expected values of functionals $Q$ of the PDE solution $u$. An easily computed example of physical interest is the pressure of the water at a given point $\mathbf{x}^* \in D$ or the $L_2$ norm of the solution $\|u\|_{L_2(D)}$. 

\subsection{Finite Elements discretisation}
\label{sec: fem}
To compute quantities of interest $Q$ of the PDE solution $u$, we adopt the finite element method to obtain numerical approximations $u_h$ of $u$ and, subsequently $Q_h$ of $Q$. We do not dwell on the practical implementation, as we exploit the Python FEniCS software \citep{alnaes_fenics_2015, logg_automated_2012} for this. Rather, we focus on theoretical bounds on the error $\|u - u_h\|_{H^1(D)}$, as this is crucial in our simulations. For simplicity, we consider zero Dirichlet boundary conditions $u=0$ on $\partial D$ for the analysis, but we expect similar results to hold with the mixed boundary conditions presented above.

To this end, consider the variational formulation of the PDE \eqref{eq: pde-model} for a given sample $k(\mathbf{x}, \omega)$:
\begin{equation*}
    \int_D k(\cdot, \omega) \nabla u \cdot \nabla v \, \text{ d}\mathbf{x} = \int_D f \, v \, \text{d}\mathbf{x}, \quad \forall v \in \mathcal{V},
\end{equation*}
where $\mathcal{V}=H^1_0(D)$ is the space of test functions. Upon discretisation, we seek solutions to this equation in a subspace $\mathcal{V}_h \subset \mathcal{V}$ spanned by piece-wise linear polynomials. This is because the low regularity of the coefficient $k$ limits the accuracy we can achieve with a higher order polynomial basis. For a more detailed implementation of finite elements for this problem see e.g. \citep{cliffe_parallel_2000} or \citep{graham_quasi-monte_2011}, and for a standard cell-centred finite volume approach see \citep{cliffe_multilevel_2011}.

The error $\|u - {u}_h\|_{H^1(D)}$ is quantified in \citep[Proposition 3.13]{charrier_finite_2013}, where ${u}_h$ is the finite element approximation obtained by applying either the midpoint rule or the trapezoidal rule when assembling the resulting linear system. In particular, this takes into account both the finite element error and the quadrature error, and $u_h$ depends on $k$ only through its values at the grid points. The result is summarised below.
\begin{Proposition}
\label{prop: discretisation-error}
    Let $u(\cdot, \omega), {u}_h(\cdot, \omega) \in H^1(D)$ be as above. Then, for $0 < s < t \leq 1$ and $0 < h < 1$:
    \begin{equation*}
        \|u - {u}_h\|_{L^p(\Omega, H^1(D))} \lesssim h^s,
    \end{equation*}
    for any $p \in [1, \infty)$ and $s \neq \frac{1}{2}$, under the assumptions that:
    \begin{enumerate}[(A)]
        \item the domain $D \subset \mathbb{R}^d$, $d = 1,2$, is a bounded, convex, Lipschitz polygonal domain. \label{assmpt: domain}
        \item $k_{\min}(\omega) \coloneqq \min_{\mathbf{x} \in \bar{D}} k(\mathbf{x}, \omega) > 0$ almost surely and $1/k_{\min} \in L^p(\Omega)$, for all $p \in [1, \infty)$. \label{assmpt: B}
        \item $k \in L^p(\Omega, C^t(\bar{D}))$, for some $0 < t \leq 1$ and for all $p \in [1, \infty)$. \label{assmpt: C}
        \item $f \in H^{t-1}(D)$. \label{assmpt: final}
    \end{enumerate}
\end{Proposition}

Note that Assumption \ref{assmpt: domain} was replaced by $\mathcal{C}^2$ bounded domains $D \subset \mathbb{R}^d$ in \citep{charrier_finite_2013}, but the result extends to polygonal domains using the results on the regularity of $u$ in \citep{teckentrup_further_2013}. Assumptions \ref{assmpt: B} and \ref{assmpt: C} are satisfied for log-normal random fields as described in section \ref{sec: pde-model}. We can extend Proposition \ref{prop: discretisation-error} to the error $\|Q-{Q}_h\|_{L^p(\Omega)}$, where $Q = \mathcal{G}(u(\cdot, \omega))$ and ${Q}_h = \mathcal{G}({u}_h(\cdot, \omega))$, for $u$ and ${u}_h$ as above, and for some bounded functional $\mathcal{G} : H^1(D) \rightarrow \mathbb{R}$. This is stated in \citep[Lemma 3.2]{teckentrup_further_2013}:
\begin{Lemma}
\label{lemma: functional-bound}
    Let $u(\cdot, \omega), {u}_h(\cdot, \omega) \in H^1(D)$ be as above. Suppose $\mathcal{G} : H^1(D) \rightarrow \mathbb{R}$ is a bounded functional satisfying the following assumption:
    \begin{enumerate}[(A)]
    \setcounter{enumi}{4}
        \item  $\mathcal{G}$ is continuously Fr\'echet differentiable, and  there exists $C_F \in L^{q}(\Omega)$, for all $q \in [1,\infty)$, such that: \label{assmpt: functional}
        \begin{equation*}
            |\overline{D_v\mathcal{G}}(u,{u}_h)| \lesssim C_F(\omega) \|v\|_{H^{1}(D)}, \quad \text{for all } v \in H_0^1(D) \quad \text{and} \quad \text{for almost all } \omega \in \Omega,
        \end{equation*}
    where:
    \begin{equation*}
        \overline{D_v \mathcal{G}}(u,{u}_h) \coloneqq \int_0^1 D_v \mathcal{G}(u + \theta({u}_h-u)) \mathrm{d}\theta.
    \end{equation*}
    \end{enumerate}
    Here, the Gateaux derivative of $\mathcal{G}$ at $\tilde{v}$ and in the direction $v$ is defined as:
    \begin{equation*}
        D_v \mathcal{G}(\tilde{v}) \coloneqq \lim_{\varepsilon \rightarrow 0} \frac{\mathcal{G}(\tilde{v} + \varepsilon v) - \mathcal{G}(\tilde{v})}{\varepsilon}, \quad \forall v,\tilde{v} \in H^1(D).
    \end{equation*}
    Then, for all $p \in [1,\infty)$:
    \begin{equation*}
        \|Q - {Q}_h\|_{L^p(\Omega)} \lesssim \|u - {u}_h\|_{L^p(\Omega, H^1(D))}.
    \end{equation*}
\end{Lemma}

Throughout the remainder of this paper, we will use the notation $u_h$ to denote the FE solution which uses any quadrature rule based on values of the coefficient $k$ at the grid points. such as the trapezoidal rule, for the linear system assembly, and $Q_h$ the resulting functional.

\subsection{Monte Carlo estimator}
\label{sec: mc-simulations} 

To be able to define an estimator of the quantity of interest $\mathbb{E}[Q]$, let us assume that, for a finite element approximation $Q_h$ of $Q$ with mesh size $h$, the following holds:
\begin{equation*}
    \mathbb{E}[Q_h] \rightarrow \mathbb{E}[Q] \text{ as } h \rightarrow 0,
\end{equation*}
with mean order of convergence $\alpha$, so that:
\begin{equation*}
    \mathbb{E}[Q_h - Q] \leq C_\alpha h^{\alpha}, \quad \alpha > 0.
\end{equation*}

A non-intrusive method for estimating $\mathbb{E}[Q]$ is Monte Carlo, which, for the application at hand, is given by the following equation:
\begin{equation}
\label{eq: MC-estimator}
    \widehat{Q}^{\text{MC}}_{h,N} \coloneqq \frac{1}{N} \sum_{i=1}^N Q_h^{(i)},
\end{equation}
where $\{Q_h^{(i)}\}_{i=1}^N$ are independently sampled from the distribution of $Q_h$. To achieve this, we first require a sample from the Gaussian field $Z$, which then yields a realisation from the coefficient $k$. For each such sample, we compute a finite element approximation $u_h$ of the solution $u$, which then allows us to compute an estimate $Q_h$ of the quantity of interest $Q$. 



Further, we assume that the cost of computing one sample of $Q_h$ satisfies:
\begin{equation}
\label{eq: cost-per-sample}
    \mathcal{C}\left(Q_h^{(i)}\right) \leq C_\gamma h^{-\gamma}, \quad \gamma > 0.
\end{equation}
Typically, the two main contributions to the cost $\mathcal{C}\left(Q_h^{(i)}\right)$ are sampling from the parameter $k(\cdot, \cdot)$, and numerically solving the PDE \eqref{eq: pde-model}. For an appropriate FFT implementation, such as FFTW \citep{frigo_design_2005}, and an optimal iterative linear solver, such as Generalized minimal residual method with incomplete LU factorisation as a preconditioner in FEniCS, a suitable choice is $\gamma \approx d$. In addition, $C_\gamma$ might depend on the parameters of the covariance function of the random field, such as $\sigma^2$, $\lambda$, and $\nu$, but it is independent of $h$. 

Let $\mathcal{C}_\varepsilon(\widehat{Q}^{\text{MC}}_{h,N})$ be the computational $\varepsilon$-cost, quantified by the number of floating point operations (FLOPS) required in order to satisfy $e(\widehat{Q}^{\text{MC}}_{h,N}) < \varepsilon$, where $e(\widehat{Q}^{\text{MC}}_{h,N})$ is the root mean squared error (RMSE) of the MC estimator. Using Eq. \eqref{eq: cost-per-sample} and under the assumption that  $\mathbb{V}[Q_h]$ is approximately constant independent of $h$, the $\varepsilon$-cost $\mathcal{C}_\varepsilon\left(\widehat{Q}^{\text{MC}}_{h,N}\right)$ associated with the estimator \eqref{eq: MC-estimator} is given by:
\begin{equation}
\label{eq: mc-cost}
    \mathcal{C}_\varepsilon\left(\widehat{Q}^{\text{MC}}_{h,N}\right) \lesssim \varepsilon^{-2-\gamma/\alpha}.
\end{equation}

Therefore, Eq. \eqref{eq: mc-cost} emphasises that the $\varepsilon$-cost associated with the MC estimator is directly affected by the parameters $\gamma$ and $\alpha$. In particular, the latter is related to the regularity of the functional under consideration, so that the more irregular the functional is, the smaller $\alpha$ is, and, consequently, the larger the cost, rendering this approach computationally expensive.

\subsection{Multilevel Monte Carlo simulations}
\label{sec: mlmc-simulations}

The fundamental principle at the basis of Multilevel Monte Carlo simulations rests on a multilevel cost reduction technique for the standard MC method \citep{giles_multilevel_2008}. This is achieved by sampling from several approximations $Q_h$ of $Q$, rather than just one. 
To this end, let $\{h_\ell: \ell = 0, \dotsc, L\}$ be a decreasing sequence in $\mathbb{Q}$ defining the MLMC levels, so that $h_0 > \dotsc > h_L \coloneqq h$. Further, suppose for simplicity that there exists an integer $r \in \mathbb{N} \setminus \{1\}$ satisfying:
\begin{equation*}
    h_{\ell-1} = r h_\ell, \quad \forall \, \ell = 1, \dotsc, L.
\end{equation*}
In other words, the MLMC levels are given, in this case, by different grid resolutions used in approximating the solution to the PDE \eqref{eq: pde-model}. For example, we can take $h_\ell = 2^{-\ell}$, so that $r = 2$.

Next, letting $Y_\ell \coloneqq Q_{h_\ell} - Q_{h_{\ell - 1}}$, $\ell = 1, \dotsc, L$, with $Y_0 \coloneqq Q_{h_0}$, the idea is to express the approximation $\mathbb{E}[Q_h]$ of the quantity of interest $\mathbb{E}[Q]$ as:
\begin{equation*}
    \mathbb{E}[Q_h] = \mathbb{E}[Q_{h_0}] + \sum_{\ell = 1}^L \mathbb{E}[Q_{h_\ell} - Q_{h_{\ell - 1}}] = \sum_{\ell = 0}^L \mathbb{E}[Y_\ell],
\end{equation*}
via the linearity of expectation. Thus, we avoid estimating $\mathbb{E}[Q_{h_\ell}]$ on level $\ell$ directly. Rather, we approximate the correction term $\mathbb{E}[Y_\ell]$ with respect to the subsequent lower level. 
Then, the MLMC estimator is defined as \citep{giles_multilevel_2008}:
\begin{equation}
\label{eq: mlmc-estimator}
    \widehat{Q}^{\text{MLMC}}_L = \widehat{Q}^{\text{MC}}_{h_0, N_0} + \sum_{\ell = 1}^L \widehat{Y}^{\text{MC}}_{\ell, N_\ell},
\end{equation}
where $\widehat{Y}^{\text{MC}}_{\ell, N_\ell}$ is the standard MC approximation of $\mathbb{E}[Y_\ell]$, $\ell = 0, \dotsc, L$ with $N_\ell$ samples, namely:
\begin{equation*}
    \widehat{Y}^{\text{MC}}_{\ell, N_\ell} \coloneqq \frac{1}{N_\ell} \sum_{i=1}^{N_\ell} \left( Q_{h_\ell}^{(i)} - Q_{h_{\ell-1}}^{(i)} \right),
\end{equation*}
the number of samples per level satisfying $N_0 > N_1 > \dotsc > N_L$. Note that it is essential to use the same underlying sample $k(\mathbf{x}, \omega^{(i)})$ when computing the difference $Q_{h_\ell}^{(i)} - Q_{h_{\ell-1}}^{(i)}$.


The number of samples $N_\ell$ can be chosen using the formula in \citep{giles_multilevel_2008}:
\begin{equation}
\label{eq: optimal-nl}
    N_\ell = \left\lceil 2 \varepsilon^{-2} \left(\sum_{\ell=0}^L \sqrt{\mathcal{C_\ell} V_\ell}\right) \sqrt{\frac{V_\ell}{\mathcal{C}_\ell}} \right\rceil, \quad \ell = 0, \dotsc, L,
\end{equation}
where $\mathcal{C}_\ell \coloneqq \mathcal{C} \left(Y_\ell^{(i)}\right)$ represents the cost of computing one sample $Y_\ell^{(i)}$ and $V_\ell \coloneqq \mathbb{V}[Y_\ell]$ denotes the variance of one such sample.

Finally, the $\varepsilon$-cost of the MLMC estimator $\mathcal{C}_\varepsilon \left(\widehat{Q}_L^{\text{MLMC}}\right)$ is quantified in \citep[Theorem 1]{cliffe_multilevel_2011}, which states the following:
\begin{Theorem}
\label{thm: the-only-theorem}
Suppose that there exist positive constants $\alpha$, $\beta$, $\gamma$ $ > 0$ such that $\alpha \geq \frac{1}{2} \min (\beta, \gamma)$, and the following hold:
\begin{enumerate}[(i)]
    \item $|\mathbb{E}[Q_{h_\ell} - Q]| \lesssim h_\ell^{\alpha}$,
    \item $\mathbb{V}[Y_\ell] \lesssim h_\ell^{\beta}$,
    \item $\mathcal{C}_\ell \lesssim h_\ell^{-\gamma}$.
\end{enumerate}
Then, for any $\varepsilon < e^{-1}$, there exists a value $L$ (and correspondingly $h \equiv h_L$) and a sequence $\{N_\ell\}_{\ell=0}^L$ such that:
\begin{equation*}
    e \left(\widehat{Q}_L^{\text{MLMC}}\right)^2 = \mathbb{E} \left[\left(\widehat{Q}_L^{\text{MLMC}} - \mathbb{E}[Q]\right)^2\right] < \varepsilon^2,
\end{equation*}
and
\begin{equation*}
    \mathcal{C}_\varepsilon \left(\widehat{Q}_L^{\text{MLMC}}\right) \lesssim 
    \begin{cases} 
    \varepsilon^{-2}, & \text{if } \beta > \gamma, \\
    \varepsilon^{-2} (\log \varepsilon)^2, & \text{if } \beta = \gamma, \\
    \varepsilon^{-2-(\gamma-\beta)/\alpha}, & \text{if } \beta < \gamma.
    \end{cases}
\end{equation*}
\end{Theorem}

A possible implementation of the MLMC estimator is proposed by Giles et al. in \citep{giles_multilevel_2008} and Cliffe et al. in \citep{cliffe_multilevel_2011}, where they adopt an adaptive approach to computing the number of levels.

\subsection{MC and MLMC with Circulant Embedding}
\label{sec: MC-MLMC-CE}

Using the observations in Section \ref{sec: ces}, we can write the MC and MLMC estimators $\widehat{Q}_h(\mathbf{Z})$, where we use the $\widehat{Q}_h(\mathbf{Z})$ notation to highlight the dependency on the Gaussian field $\mathbf{Z}$, as follows:
\begin{equation*}
    \widehat{Q}_{h,N}^{\text{MC}}(\mathbf{Z}) = \frac{1}{N}\sum_{i=1}^N Q_h((G \Lambda^{\frac{1}{2}}\boldsymbol{\xi}^{(i)})_R), 
\end{equation*}
\begin{equation*}
    \widehat{Q}^{\text{MLMC}}_L = \frac{1}{N_0} \sum_{i=1}^{N_0} Q_{h_0}((G \Lambda^{\frac{1}{2}}\boldsymbol{\xi}^{(i,0)})_R) + \sum_{\ell=1}^L \left(\frac{1}{N_\ell} \sum_{i=1}^{N_\ell}\left(Q_{h_\ell}((G \Lambda^{\frac{1}{2}} \boldsymbol{\xi}^{(i,\ell)})_R) - Q_{h_{\ell-1}}((G \Lambda^{\frac{1}{2}} \boldsymbol{\xi}^{(i,\ell)})_R)\right) \right),
\end{equation*}
where $\{\boldsymbol{\xi}^{(i,\ell)}\}$ is a collection of IID standard Gaussian random vectors. Here, we use the same discretisation mesh $\mathcal{T}$ of the domain $D$ with mesh size $h \equiv \frac{1}{m_i}$, $i = 1, \dotsc, d$, for sampling using the circulant embedding scheme and for computing the finite element approximation $Q_h$.

Thus, the following algorithm is adapted from \citep[Section 5.3]{graham_quasi-monte_2011} and it outlines the steps required for estimating $\widehat{Y}^{\text{MC}}_{\ell, N_\ell}$ in the MLMC method on level $\ell$, where samples of $k(\mathbf{x}, \cdot)$ are obtained using circulant embedding:
\begin{Algorithm}
\label{alg: MLMC-and-CE}
Computes an MC estimate of $\mathbb{E}[Y_\ell]$ on a level $\ell$ using the circulant embedding method to sample from $k(\cdot, \cdot)$ given the number of samples $N_\ell$.
\begin{enumerate}
    \item Construct covariance matrix $R$ from discretisation mesh $\mathcal{T}$ and covariance function $C$.
    \item Embed $R$ into a $4m_1 m_2 \times 4m_1 m_2$ block circulant with circulant blocks matrix $S$.
    \item Compute ${\Lambda_j} = (\sqrt{4m_1 m_2} F \mathbf{s})_{{j}}, {j = 1, \dotsc, 4m_1 m_2}$, where $\mathbf{s}$ is the first {column} of $S$, and $F$ is the 2D Fourier matrix.
    \begin{enumerate}
        \item Check whether the eigenvalues real and positive. If yes, go to (4). If not, go to (3b).
        \item Pad the embedding matrix $S$ to obtain a symmetric positive definite matrix using the smooth periodisation technique and the bound \eqref{eq: padding-bound}, and compute updated $\Lambda_j$.
    \end{enumerate}
    \item For each MC iteration on level $\ell$:
        \begin{enumerate}
            \item Generate $\boldsymbol{\xi} \sim \mathcal{N}(0, I)$.
            \item Evaluate $\mathbf{w} \coloneqq \boldsymbol{\Lambda} \odot \boldsymbol{\xi}$, {where $\boldsymbol{\Lambda}$ is the vector whose entries are the eigenvalues $\Lambda_j$}.
            \item Compute $\mathbf{v} \coloneqq F \mathbf{w}$.
            \item Take $\mathbf{u} \coloneqq \Re(\mathbf{v}) + \Im(\mathbf{v})$.
            \item Set $\mathbf{Z}_\ell \coloneqq (\mathbf{u})_{R_\ell}$ and $\mathbf{Z}_{\ell-1} \coloneqq (\mathbf{u})_{R_{\ell-1}}$.
            \item Let $\mathbf{k}_\ell \coloneqq \exp(\mathbf{Z}_\ell)$ and $\mathbf{k}_{\ell-1} \coloneqq \exp(\mathbf{Z}_{\ell-1})$.
            \item Solve Eq. \eqref{eq: pde-model} in turn for $\mathbf{k}_\ell$ and $\mathbf{k}_{\ell-1}$.
        \end{enumerate}
\end{enumerate}
\end{Algorithm}
\noindent Here, $\odot$ denotes the Hadamard product, so that, for $\mathbf{a}, \mathbf{b} \in \mathbb{R}^d$, $\{\mathbf{a} \odot \mathbf{b}\}_j = \mathbf{a}_j \mathbf{b}_j$, $j = 1, \dotsc, d$.

It is worth mentioning that the 2D discrete Fourier transform in Step 3 and Step 4(c) can be computed using an FFT implementation, such as FFTW \citep{frigo_design_2005}, significantly reducing the computational complexity. In addition, $S$ is a circulant matrix and, therefore, uniquely defined by its first row or column. As a consequence, Step 1 and Step 2 can be altered to only construct the first row or column of $S$, rather than the full matrix, diminishing the memory requirements.

\subsection{MC and MLMC with Smoothed Circulant Embedding}

One argument which supports the computational efficiency of the MLMC estimator \eqref{eq: mlmc-estimator} over the standard MC approximation \eqref{eq: MC-estimator} is that the coarsest level is invariant for any imposed accuracy $\varepsilon$, and, consequently, the cost $\mathcal{C}_0$ does not increase as $\varepsilon \rightarrow 0$. Generally, a minimum value $h_0$ is necessary to provide a basic level of detail on the problem. This is selected depending on the regularity of the solution to the PDE \eqref{eq: pde-model}. While this can render the MLMC estimator more costly for large accuracies $\varepsilon$, it still leads to significant reductions in cost as $\varepsilon \rightarrow 0$.

In some applications, the choice of coarsest mesh is limited by the length scale $\lambda$ and the smoothness parameter $\nu$ of the random field $k(\cdot, \cdot)$. This is illustrated in Figure \ref{fig: smoothing-motivation-matern} for the functional $Q = u\left(\frac{7}{15}, \frac{7}{15}\right)$ of the solution to the PDE \eqref{eq: pde-model}, where we use the Mat\'ern covariance for the random parameter $k(\cdot, \cdot)$ with $\sigma^2=1$, $\nu=1.5$ and $\lambda=0.03$. In particular, the plot highlights that on the first levels, which correspond to coarse meshes, the variance $\mathbb{V}[Y_\ell]$ of the difference $Q_{h_\ell} - Q_{h_{\ell-1}}$ is larger than then the variance $\mathbb{V}[Q_{h_\ell}]$ of the quantity of interest. Hence, for these levels, the contributions to the cost $\mathcal{C}_\varepsilon\left(\widehat{Q}_L^{\text{MLMC}}\right)$ will actually be more significant than those using standard MC, rendering the MLMC approach more expensive.

\begin{figure}
    \centering
    \includegraphics[width=0.6\textwidth]{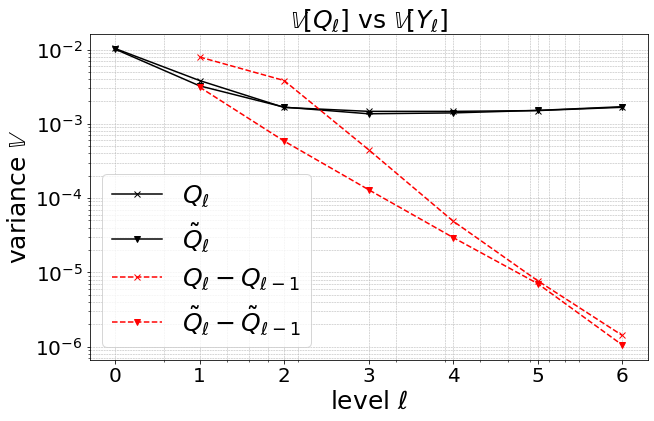}
    \caption{Plot of variance decay for the quantity of interest $Q = {u(\mathbf{x}^*)}$ using the Mat\'ern covariance function for the random coefficient with $\sigma^2=1$, $\nu=1.5$ and $\lambda=0.03$.}
    \label{fig: smoothing-motivation-matern}
\end{figure}

For this reason, Cliffe et al. \citep{cliffe_multilevel_2011} make the following observation at the end of Section 4.1:
\begin{equation}
\label{eq: mlmc-h0-bound-exp-covariance}
    h_0 \leq \lambda,
\end{equation}
for the covariance function in Eq. \eqref{eq: p-norm-cov-function}. For the Mat\'ern case, we use the heuristically derived relationship between the correlation length $\rho$, smoothness parameter $\nu$ and length scale $\lambda$ \citep{lindgren_explicit_2011}:
\begin{equation*}
    \rho = \sqrt{8\nu} \lambda,
\end{equation*}
so that a suitable estimate for the coarsest mesh in the MLMC estimator is:
\begin{equation}
\label{eq: mlmc-h0-bound-matern-covariance}
    h_0 \leq \sqrt{8 \nu} \lambda.
\end{equation}
This corresponds to the point where the two lines of $\mathbb{V}[Y_\ell]$ and $\mathbb{V}[Q_{h_\ell}]$ intersect in Figure \ref{fig: smoothing-motivation-matern}, and limits severely the computational gains prompted by MLMC for small $\lambda$ and $\nu$, as it restricts the number of (computationally cheap) levels we can include in the MLMC estimator. {(Recall that MLMC crucially uses that $\mathbb{V}[Y_\ell]$ is small, so that $\mathbb{E}[Y_\ell]$ can be estimated accurately using a small number of samples.)}

One solution to this problem was proposed in \citep{teckentrup_further_2013} where the KL-expansion is used to sample from the random coefficient. To mitigate the aforementioned issue, Teckentrup et al. suggest a \textit{level-dependent} truncation of the KL-expansion, where the most oscillatory terms in the sum are neglected to obtain a more accurate approximation of the random field on coarser levels. This then allows to choose $h_0$ independent of $\lambda$, leading to noteworthy reductions in the cost. Similar ideas on level-dependent truncation have been used in \citep{gittelson_multi-level_2013, schwab_multilevel_2023}.

The goal is to replicate this idea in the context of circulant embedding methods. It is worth highlighting that the truncated KL-expansion is an approximation of the random field, and so the accuracy of the estimate is directly affected by where the sum is truncated. On the other hand, circulant embedding methods are exact on the discretisation mesh, and so we can no longer distinguish different levels of approximation of the random field. Rather, we limit the extent of fluctuations in $k(\mathbf{x}, \omega)$ on the mesh considered on each level of the MLMC estimator by discarding highly oscillatory terms corresponding to the smallest eigenvalues, as outlined in Section \ref{sec: smoothing}. In doing so, consecutive differences in the definition of $\widehat Q^{\text{MLMC}}_{L}$ have smaller variance. This leads to a choice of $h_0$ independent of the parameters $\lambda$ and $\nu$.

{It is worth highlighting that the bounds in Eq. \eqref{eq: mlmc-h0-bound-exp-covariance} and \eqref{eq: mlmc-h0-bound-matern-covariance} indicate that the smaller $\lambda$ is, the finer the coarsest mesh in standard MLMC is. In addition, in order to achieve small accuracies, we use more levels in the MLMC estimator with smaller mesh sizes which eventually dominate the computational cost, regardless of the coarsest mesh. Hence, the importance of the smoothing technique is mostly evident for small correlation lengths $\lambda$ and large errors $\varepsilon$.}

To justify using the CES approach in the MLMC estimator for the quantity of interest $\mathbb{E}[Q]$ from the PDE model \eqref{eq: pde-model}, the following theorem provides an estimate of the error introduced by smoothing the random field samples on a given level $\ell$:
\begin{Theorem}
\label{thm: error-functional-sample}
    Let  $Z_\mathcal{T}$ and $\tilde{Z}_\mathcal{T}$ be piece-wise linear interpolants of the non-smoothed and smoothed discrete representations $\mathbf{Z}$ and $\tilde{\mathbf{Z}}$, respectively, on a given grid $\mathcal{T}$ of the domain $D$. Let $\tilde{k} = \exp(\tilde{Z}_\mathcal{T})$ and denote by $\tilde{u}_h$ the corresponding finite element solution. Let $\mathcal{G} : H^1(D) \rightarrow \mathbb{R}$ be a bounded functional which satisfies assumption \eqref{assmpt: functional} in Lemma \ref{lemma: functional-bound} with $u_h$ and $\tilde u_h$ instead of $u$ and $u_h$. {If $\tau$ is such that $\mathbb{E}[\|\mathbf{Z}(\omega) - \tilde{\mathbf{Z}}(\omega)\|^p_\infty] \rightarrow 0$ as $h \rightarrow 0$}, then:
    \begin{equation*}
        \mathbb{E}\left[|Q _h - \Tilde{Q}_h|^p\right] \lesssim \mathbb{E}\left[\|Z_\mathcal{T} - \Tilde{Z}_\mathcal{T}\|^{pq}_{C^0(\bar{D})}\right]^{\frac{1}{q}},
    \end{equation*}
    for any $p,q \in [1, \infty)$.
\end{Theorem}
\begin{proof}
    From Lemma \ref{lemma: functional-bound}, we have:
    \begin{equation*}
        \mathbb{E}\left[|Q_h(\cdot, \omega) - \Tilde{Q}_h(\cdot, \omega))|^p\right] \lesssim \mathbb{E}\left[\|u_h(\cdot, \omega) - \Tilde{u}_h(\cdot, \omega)\|^p_{H^1(D)}\right].
    \end{equation*}
    Then, from \citep[Lemma 4.8]{dodwell_hierarchical_2015}, and since the finite element solution computed with a nodal quadrature rule only depends on the values of the coefficient on the grid $\mathcal T$, we know that:
    \begin{equation*}
        \mathbb{E}\left[\|u_h(\cdot, \omega) - \Tilde{u}_h(\cdot, \omega)\|^p_{H^1(D)}\right] \lesssim \mathbb{E} \left[{\frac{1}{\kappa(\omega)^p \tilde{\kappa}(\omega)^p}} \| k(\cdot, \omega) - \Tilde{k}(\cdot, \omega)\|^p_{C^0(\Bar{D})}\right],
    \end{equation*}
    where $k= \exp(Z_\mathcal{T})$ {and $\kappa(\omega) = \min_{x \in \bar{D}} k(x, \omega)$. Applying H\"older's inequality, the right-hand side becomes:
    \begin{equation}
    \label{eq: theorem-3-proof-3}
        \mathbb{E} \left[\frac{1}{\kappa(\omega)^p \tilde{\kappa}(\omega)^p}\| k(\cdot, \omega) - \Tilde{k}(\cdot, \omega)\|^p_{C^0(\Bar{D})}\right] \leq \mathbb{E}\left[\frac{1}{\kappa(\omega)^{pq_1} \tilde{\kappa}(\omega)^{pq_1}}\right]^\frac{1}{q_1} \mathbb{E}[\| k(\cdot, \omega) - \Tilde{k}(\cdot, \omega)\|^{pr_1}_{C^0(\Bar{D})}]^{\frac{1}{r_1}},
    \end{equation}
    for some $q_1, r_1 \in [1, \infty)$ such that $\frac{1}{q_1} + \frac{1}{r_1} = 1$.}

    {We know from \citep[Prop. 3.10]{charrier_strong_2012} that $\mathbb{E}[\bar{\kappa}(\omega)^{-p}]$ is bounded independent of $h$, where $\bar{\kappa} = \min_{x \in \bar{D}} \bar{k}(x, \omega)$ with $\bar{k}$ denoting the continuous exponential field. Then, since $k = \exp(Z_\mathcal{T})$ is the linear interpolant of $\bar{k}$, it follows that $\mathbb{E}[\kappa(\omega)^{-p}] \geq \mathbb{E}[\bar{\kappa}(\omega)^{-p}]$, so that $\mathbb{E}[\kappa(\omega)^{-p}]$ is bounded independent of $h$ . To show that this is the case for $\mathbb{E}[\tilde{\kappa}(\omega)^{-p}]$, note that:
     \begin{align*}
        \mathbb{E}[\|\tilde{Z}_\mathcal{T}(\cdot, \omega)\|^p_{C^0(\bar{D})}] &\leq \mathbb{E}[(\|Z_\mathcal{T}(\cdot, \omega)\|_{C^0(\bar{D})} + \|Z_\mathcal{T}(\cdot, \omega) - \tilde{Z}_\mathcal{T}(\cdot, \omega)\|_{C^0(\bar{D})})^p] \\
        &\leq 2^{p-1} \mathbb{E}[\|Z_\mathcal{T}(\cdot, \omega)\|^p_{C^0(\bar{D})}] + 2^{p-1} \mathbb{E}[\|Z_\mathcal{T}(\cdot, \omega) - \tilde{Z}_\mathcal{T}(\cdot, \omega)\|^p_{C^0(\bar{D})}] \\
        &= 2^{p-1} \mathbb{E}[\|Z_\mathcal{T}(\cdot, \omega)\|^p_{C^0(\bar{D})}] + 2^{p-1} \mathbb{E}[\|\mathbf{Z}(\omega) - \tilde{\mathbf{Z}}(\omega)\|^p_\infty],
    \end{align*}
    where we used the fact that:
    \begin{equation}
    \label{eq: theorem-3-proof-2}
        |a+b|^t \leq 2^{t-1} (|a|^t + |b|^t), \quad \forall a, b \in \mathbb{R} \text{ and } t \in [1, \infty),
    \end{equation}
    and that $Z_\mathcal{T}$ and $\tilde{Z}_\mathcal{T}$ are piece-wise linear interpolants of $\mathbf{Z}$ and $\tilde{\mathbf{Z}}$.}
    
    {Now, we know from \cite[Proposition 3.8]{charrier_strong_2012} that: 
    $$\mathbb{E}[\|\Bar{Z}(\cdot, \omega)\|^p_{C^0(\bar{D})}] < C_1,$$ 
    where $\Bar{Z}(\mathbf{x}, \omega)$ is the continuous Gaussian field such that $\Bar{k}(\mathbf{x}, \omega) = \exp(\Bar{Z}(\mathbf{x}, \omega))$. Since $Z_\mathcal{T}$ is the linear interpolant of $\Bar{Z}$, it follows that:
    $$\mathbb{E}[\|Z_\mathcal{T}(\cdot, \omega)\|^p_{C^0(\bar{D})}] < \mathbb{E}[\|\Bar{Z}(\cdot, \omega)\|^p_{C^0(\bar{D})}] < C_1,$$ 
    independently of $h$. From the assumption on $\tau$ we have that $\mathbb{E}[\|\mathbf{Z}(\omega) - \tilde{\mathbf{Z}}(\omega)\|^p_\infty] \rightarrow 0$ as $h \rightarrow 0$. Hence, we have that:
    $$\mathbb{E}[\|\tilde{Z}_\mathcal{T}(\cdot, \omega)\|^p_{C^0(\bar{D})}] < C_3$$ 
    independent of $h$, so that:
    $$\mathbb{E}[\exp(\|\Tilde{Z}_\mathcal{T}(\cdot, \omega)\|^p_{C^0(\bar{D})})] < C_4$$ independent of $h$. Further:
    \begin{align*}
        \mathbb{E}\left[\frac{1}{\tilde{\kappa}(\omega)^p}\right] &=  \mathbb{E}\left[\frac{1}{(\min_{x \in \bar{D}}\tilde{k}(x, \omega))^p}\right] = \mathbb{E}\left[\frac{1}{(\min_{x \in \bar{D}}\exp(\Tilde{Z}_\mathcal{T}(x, \omega)))^p}\right] \\
        &= \mathbb{E}\left[\frac{1}{\exp(\min_{x \in \bar{D}}\Tilde{Z}_\mathcal{T}(x, \omega))^p}\right] \leq \mathbb{E}[\exp \|\tilde{Z}(x, \omega)\|^p_{C^0(\bar{D})}] < C_4
    \end{align*}
    independent of $h$ from above. Hence, $\mathbb{E}\left[\frac{1}{\kappa(\omega)^{pq} \tilde{\kappa}(\omega)^{pq}}\right]^\frac{1}{q}$ is bounded independent of $h$.
    }
    
    {For the second term in Eq. \eqref{eq: theorem-3-proof-3}}, using the inequality $|e^x - e^y| \leq (e^x + e^y) |x-y|, \forall x, y \in \mathbb{R}$ and H\"older's inequality, we can derive:
    \begin{multline*}
        \mathbb{E}\left[\|\exp(Z_\mathcal{T}(\cdot, \omega)) - \exp(\Tilde{Z}_\mathcal{T}(\cdot, \omega))\|^{{pr_1}}_{C^0(\Bar{D})}\right]^{{\frac{1}{r_1}}} \leq \\ \mathbb{E} \left[\|\exp(Z_\mathcal{T}(\cdot, \omega)) + \exp(\Tilde{Z}_\mathcal{T}(\cdot, \omega))\|^{{p r_1 q_2}}_{C^0(\Bar{D})} \right]^{{\frac{1}{r_1 q_2}}} \mathbb{E}\left[\|Z_\mathcal{T}(\cdot, \omega) - \Tilde{Z}_\mathcal{T}(\cdot, \omega) \|^{{p r_1 r_2}}_{C^0(\Bar{D})}\right]^{{\frac{1}{r_1 r_2}}},
    \end{multline*}
    where {$q_2, r_2 \in [1, \infty]$} are H\"older exponents. 

    {Further, we wish to prove that the first term in the inequality on the right-hand side is bounded independent of the discretisation parameter $h$. To this end, using the triangle inequality and Eq. \eqref{eq: theorem-3-proof-2}, we obtain:
    \begin{align}
    \label{eq: theorem-3-proof-1}
        &\mathbb{E} \left[\|\exp(Z_\mathcal{T}(\cdot, \omega)) + \exp(\Tilde{Z}_\mathcal{T}(\cdot, \omega))\|^{{p r_1 q_2}}_{C^0(\Bar{D})} \right]^{{\frac{1}{r_1 q_2}}} \nonumber \\
        &\qquad \leq 2^{\frac{p r_1 q_2 - 1}{r1 q_2}} \left(\mathbb{E}\left[\|\exp(Z_\mathcal{T}(\cdot, \omega))\|^{{p r_1 q_2}}_{C^0(\Bar{D})}\right] + \mathbb{E}\left[\|\exp(\Tilde{Z}_\mathcal{T}(\cdot, \omega))\|^{{p r_1 q_2}}_{C^0(\Bar{D})}\right] \right)^{\frac{1}{r_1 q_2}} \nonumber \\ 
        &\qquad \leq 2^{\frac{p r_1 q_2 - 1}{r1 q_2}} \left(\mathbb{E}\left[\|\exp(Z_\mathcal{T}(\cdot, \omega))\|^{{p r_1 q_2}}_{C^0(\Bar{D})}\right]^\frac{1}{r_1 q_2} + \mathbb{E}\left[\|\exp(\Tilde{Z}_\mathcal{T}(\cdot, \omega))\|^{{p r_1 q_2}}_{C^0(\Bar{D})}\right]^{\frac{1}{r_1 q_2}} \right) \nonumber \\ 
        &\qquad = 2^{\frac{p r_1 q_2 - 1}{r1 q_2}} \left(\mathbb{E}\left[\exp(\|Z_\mathcal{T}(\cdot, \omega)\|)^{{p r_1 q_2}}_{C^0(\Bar{D})}\right]^\frac{1}{r_1 q_2} + \mathbb{E}\left[\exp(\|\Tilde{Z}_\mathcal{T}(\cdot, \omega)\|)^{{p r_1 q_2}}_{C^0(\Bar{D})}\right]^{\frac{1}{r_1 q_2}} \right)
    \end{align}
    In the last inequality, we used the fact that $(|a| + |b|)^p \leq |a|^p + |b|^p$ for $a, b \in \mathbb{R}$ and $0 \leq p \leq 1$. For the first term in Eq. \eqref{eq: theorem-3-proof-1}, we know from above that $\mathbb{E}[\|Z_\mathcal{T}(\cdot, \omega)\|^p_{C^0(\Bar{D})}] < C_1$ independent of $h$, so that $\mathbb{E}[\exp(\|Z_\mathcal{T}(\cdot, \omega)\|)^p_{C^0(\Bar{D})}] < C_2$ independent of $h$.}

    {Putting everything together, we obtain that $\mathbb{E} \left[\|\exp(Z_\mathcal{T}(\cdot, \omega)) + \exp(\Tilde{Z}_\mathcal{T}(\cdot, \omega))\|^{{p r_1 q_2}}_{C^0(\Bar{D})} \right]^{{\frac{1}{r_1 q_2}}}$ is bounded independent of h, which concludes the proof.}
    
\end{proof}

To obtain a bound for the error between $Q_h$ and $\tilde{Q}_h$ in terms of the number of dropped eigenvalues $k$, we combine Theorems \ref{thm: sample-error} and \ref{thm: error-functional-sample} to infer the following Corollary:
\begin{Corollary}
\label{cor: error-functional-egnv}
Let the assumptions of Theorems \ref{thm: sample-error} and \ref{thm: error-functional-sample} hold. Then, for any $p \in [1, \infty)$:
    \begin{equation*}
        \mathbb{E}\left[|Q_h - \tilde{Q}_h|^p\right] \lesssim s^{-\frac{p}{2}} \left(\Lambda_{{s-{\tau}+1}}^{{\mathrm{ord}}}\right)^\frac{p}{2} \, {\tau}^p,
    \end{equation*}
{where $s$ if the size of the embedding matrix after padding, so that $s= 2^d \prod_{i=1}^d (m_i + J_i)$, for discretisation parameters $m_i, i=1, \dotsc, d$ and padding values $J_i, i=1, \dotsc, d$.}
\end{Corollary}
\begin{proof}
 This follows directly from Theorems \ref{thm: sample-error} and \ref{thm: error-functional-sample}, together with the observation that since $Z_\mathcal{T}$ and $\tilde{Z}_\mathcal{T}$ are piece-wise linear interpolants of $\mathbf{Z}$ and $\tilde{\mathbf{Z}}$, respectively, we have that:
 \begin{equation*}
	\|Z_\mathcal{T}(\cdot, \omega) - \tilde{Z}_\mathcal{T}(\cdot, \omega)\|_{C^0(\bar{D})} = \|\mathbf{Z}(\omega) - \tilde{\mathbf{Z}}(\omega)\|_\infty.
    \end{equation*}
\end{proof}

This immediately leads to the following Corollary:
\begin{Corollary}
\label{cor: error-1-norm-cov}
    Let $Q_h$ and $\tilde{Q}_h$ be as in Corollary \ref{cor: error-functional-egnv}, and let $C$ be the separable exponential covariance \eqref{eq: p-norm-cov-function}. Then, for any $p \in [1, \infty)$:
    \begin{equation*}
        \mathbb{E}\left[|Q_h - \tilde{Q}_h|^p\right] \lesssim (s - {\tau} {+1})^{-p} \, {\tau}^p,
    \end{equation*}
{where $s$ if the size of the embedding matrix, so that $s = 2^d \prod_{i=1}^d m_i$, for discretisation parameters $m_i, i=1, \dotsc, d$.}
\end{Corollary}
\begin{proof}
    This follows immediately from Corollary \ref{cor: error-functional-egnv} and Lemma \ref{lemma: expo-egnv-decay}, where we used $\prod_{i=1}^d m_i = \frac{s}{2^d}$.
\end{proof}

For the Mat\'ern covariance \eqref{eq: matern-cov-function} where we use smooth periodisation to ensure the positive-definiteness of the (extended) embedding matrix $S$, we can directly use Eq. \eqref{eq: egnv-matern-cov-decay} to infer:

\begin{Corollary}
\label{cor: error-matern-cov}
    Let $Q_h$ and $\tilde{Q}_h$ be as in Corollary \ref{cor: error-functional-egnv}, and let $C$ be the Mat\'ern covariance \eqref{eq: matern-cov-function} satisfying the assumptions in Lemma \ref{thm: matern-egnv-decay}. Then, for any $p \in [1, \infty)$:
    \begin{equation*}
        \mathbb{E}\left[|Q_h - \tilde{Q}_h|^p\right] \lesssim s^{-\frac{p}{2}} {\left(\prod_{i=1}^d m_i\right)^{\frac{p}{2}}} (s - {\tau} {+1})^{-\frac{p}{2}(1+\frac{2\nu}{d})} \, {\tau}^p,
    \end{equation*}
{where $s$ if the size of the embedding matrix after padding, so that $s = 2^d \prod_{i=1}^d (m_i + J_i)$, for discretisation parameters $m_i, i=1, \dotsc, d$ and padding values $J_i, i=1, \dotsc, d$.}
\end{Corollary}
\begin{proof}
    This follows immediately from Corollary \ref{cor: error-functional-egnv} and Lemma \ref{thm: matern-egnv-decay}. 
\end{proof}

The results in Corollaries \ref{cor: error-1-norm-cov} and \ref{cor: error-matern-cov} quantify the error introduced by smoothing the random field $Z$ for the application at hand. In addition, they also provide a rule for choosing the number of eigenvalues ${\tau}_\ell$ to drop on each level in the MLMC estimator. More specifically, observe that:
\begin{align*}
    \mathbb{E}[|Q_{h_\ell} - \tilde{Q}_{h_{\ell-1}}|] &\leq \mathbb{E}[|Q_{h_\ell} - Q_{h_{\ell-1}}|] + \mathbb{E}[|Q_{h_{\ell-1}} - \tilde{Q}_{h_{\ell-1}}|] \nonumber \\
    &\leq C_\alpha h_\ell^\alpha + \mathbb{E}[|Q_{h_{\ell-1}} - \tilde{Q}_{h_{\ell-1}}|].
\end{align*}
For the exponential covariance with $p=1$, we obtain:
\begin{equation*}
    \mathbb{E}[|Q_{h_\ell} - \tilde{Q}_{h_{\ell-1}}|] \leq C_\alpha h_\ell^\alpha + C_s (s_\ell - {\tau}_\ell {+1})^{-1} \, {\tau}_\ell,
\end{equation*}
Equating the decay rate of the discretisation error with the smoothing error, and using that, in this case, no padding is needed, so that $s_\ell = 4 h_\ell^{-2}$, we find the following expression for ${\tau}_\ell$:
\begin{equation}
\label{eq: k-values-p-norm-cov}
    {\tau_\ell = \frac{s_\ell + 1}{\frac{C_s}{2^\alpha C_\alpha} s_\ell^{\frac{\alpha}{2}}+1}},
\end{equation}
{which satisfies $1 \leq \tau_\ell \leq s_\ell$ if and only if $\alpha \leq 2$. This is actually in accordance with our theory since, for the separable exponential covariance function, we already have $\alpha \leq 2$ \citep{teckentrup_multilevel_2012}.}

Analogously, for the Mat\'ern covariance function, we recover:
\begin{equation*}
    \mathbb{E}[|Q_{h_\ell} - \tilde{Q}_{h_{\ell-1}}|] \leq C_\alpha h_\ell^{\alpha} + C_s s_\ell^{-\frac{1}{2}} {(h_\ell)^{-1}} (s_\ell - {\tau}_\ell {+1})^{-\frac{1}{2}(1+\nu)} \, {\tau}_\ell.
\end{equation*}
Note that if the smooth periodisation extension in Section \ref{sec: periodisation} is applied to ensure the positive definiteness of the embedding matrix, we have $s_\ell = 4(h_\ell^{-1}+J_\ell)^2$, where $J_\ell$ is the smoothing parameter. In this case, we determine the following rule for choosing ${\tau}_\ell$:
\begin{equation}
\label{eq: tau-values-matern}
    \left(4(h_\ell^{-1}+J_\ell \right)^2 - {\tau}_\ell {+1})^{-\frac{1}{2}(1+\nu)} {\tau}_\ell = \frac{2 C_\alpha}{C_s} h_\ell^{\alpha {+1}} (h_\ell^{-1}+J_\ell).
\end{equation}
In practice, we solve this equation numerically for ${\tau}_\ell$, with negligible computational cost.

In the context of the MLMC estimator, it is necessary to highlight that in order to obtain an accurate approximation of the solution, only the random field samples on the levels preceding the finest level $L$ {should be smoothed}. This is because the approximation on level $L$ in the telescopic sum must always use a complete representation of the random field, in order to avoid introducing an additional error in the final estimate. In other words, Eq. \eqref{eq: mlmc-estimator} becomes:
\begin{equation}
\label{eq: mlmc-estimator-smoothed}
    \widehat{Q}^{\text{MLMC-CES}}_L = \sum_{\ell = 0}^{L-1} \frac{1}{N_\ell} \sum_{i=1}^{N_\ell} \tilde{Y}_\ell^{(i)} + \frac{1}{N_L} \sum_{i=1}^{N_L} Q_{h_L}^{(i)} - \tilde{Q}_{h_{L-1}}^{(i)} , 
\end{equation}
where $\tilde{Y}_\ell^{(i)} = \tilde{Q}_{h_\ell}^{(i)} - \tilde{Q}_{h_{\ell-1}}^{(i)}$, and $\tilde{Y}_0^{(i)} = \tilde{Q}_{h_0}^{(i)}$. An alternative implementation involves first finding the level $\ell^* < L$ up to which smoothing is necessary using Eq. \eqref{eq: mlmc-h0-bound-exp-covariance} or \eqref{eq: mlmc-h0-bound-matern-covariance}, and subsequently applying the smoothing technique only for $\ell = 0, \dotsc, \ell^*$. However, in our numerical experiments below, this did not appear to be favourable.

Hence, we adapt Algorithm \ref{alg: MLMC-and-CE} to estimate $\tilde{Y}^{\text{MC}}_{\ell, N_\ell}$ in the MLMC method \eqref{eq: mlmc-estimator-smoothed} by altering step 4(b) as follows:
\begin{Algorithm}
\label{alg: MLMC-and-CES}
Computes an MC estimate of $\mathbb{E}[Y_\ell]$ on a level $\ell$ using the smoothed circulant embedding method to sample from $\tilde{k}(\cdot, \cdot)$ given the number of samples $N_\ell$ and the number of eigenvalues to drop ${\tau}_\ell$.
\begin{enumerate}
    \item[{4(b)}] {Compute $\{\Lambda_j^{\mathrm{ord}}\}_{j=1}^{s_\ell}$ and set $\{\Lambda_j^\mathrm{ord}\}_{j=s_\ell-{\tau}_\ell+1}^{s_\ell} = 0$.} Let $\tilde{\boldsymbol{\Lambda}}$ denote the resulting vector. Evaluate $\mathbf{w} \coloneqq \tilde{\boldsymbol{\Lambda}} \odot \boldsymbol{\xi}$.
\end{enumerate}
\end{Algorithm}
{\noindent Note that here we are computing a permutation $\pi : \{1, 2, \dotsc, s_\ell\} \rightarrow \{1, 2, \dotsc, s_\ell\}$ such that the ordered eigenvalues $\{\Lambda_j^{\mathrm{ord}}\}_{j=1}^{s_\ell}$ are {given by} $\{\Lambda_{\pi(j)}\}_{j=1}^{s_\ell}$. The vector $\tilde{\boldsymbol{\Lambda}}$ is then obtained by applying the inverse permutation $\pi^{-1}$ to the truncated eigenvalues.}

\section{Numerical Results}
\label{sec: numerical-results}

In this section, we propose a series of numerical experiments which are intended to establish the performance of the smoothed circulant embedding method. We focus on MLMC applied to the PDE \eqref{eq: pde-model} with a log-normal random field $k(\mathbf{x}, \omega)$ on $D=(0,1)^2$ and for $f \equiv 1$. We also offer a brief comparison with the level-dependent truncation of the KL-expansion in \citep{teckentrup_further_2013}. {The code used to carry out these experiments can be found on the Github repository \url{https://github.com/anastasia5598/circulant-embedding-with-smoothing}}.

We consider two different quantities of interest, namely $Q = u\left(\frac{7}{15}, \frac{7}{15}\right)$ and $Q =  \|u(\cdot, \omega)\|_{L^2(D)}$. In fact, the point evaluation of the solution to the PDE \eqref{eq: pde-model} is not a bounded functional on the Sobolev space $H_0^1(D)$ in domains with more than one dimension. Nonetheless, convergence rates of the form in Proposition \ref{prop: discretisation-error} can still be obtained as shown in \citep{teckentrup_multilevel_2012}.

In the absence of an analytical solution to the two-dimensional PDE in Eq. \eqref{eq: pde-model}, we must perform a set of additional steps in order to obtain an estimate for the MSEs $e\left(\widehat{Q}^{\text{MC}}_{h,N}\right)$ and $e\left(\widehat{Q}^{\text{MLMC}}_{L}\right)$. In particular, a term which arises in both errors is the bias $(\mathbb{E}[Q_h - Q])^2$, which can be computed by first assuming that the decay in $\lvert \mathbb{E}[Q_h-Q] \rvert$ is monotonic when $h \leq h'$, for some $h' \in \mathbb{Q}$, and that:
\begin{equation*}
    |\mathbb{E}[Q_h-Q]| = C_\alpha \, h^{\alpha}, \quad \alpha > 0.
\end{equation*}
Then, using Richardson's extrapolation \citep{richardson_approximate_1911} with extrapolant $Q_{2h}$, we obtain:
\begin{equation*}
    |\mathbb{E} \left[Q_h - Q \right]| = \left| \frac{\mathbb{E}\left[Q_{h} - Q_{2h}\right]}{1-2^{\alpha}} \right|,
\end{equation*}
which we can now use to compute the bias numerically.
However, it is important to distinguish that, when smoothing is introduced, the extrapolant used in Richardson's extrapolation is $\tilde{Q}_{2h}$. Under the following assumption:
\begin{equation*}
    |\mathbb{E}[\tilde{Q}_{h}-Q]| = \tilde{C}_\alpha \, h^{\alpha}, \quad \alpha > 0,
\end{equation*}
we can alter the above derivation accordingly to obtain: 
\begin{equation*}
    |\mathbb{E} \left[Q_h - Q \right]| = \left| \frac{\mathbb{E}[Q_{h} - \tilde{Q}_{2h}]}{1 - \frac{\tilde{C}_\alpha}{C_\alpha} 2^{\alpha}}\right|.
\end{equation*}

While \citep{teckentrup_further_2013} gives theoretical values for $\alpha$ and $\beta$ for the functionals under consideration, here we determine these numerically by estimating $|\mathbb{E}\left[Q_{h} - Q_{2h}\right]|$ on consecutively finer meshes, and computing the best linear fit through these approximations on a log scale. In particular, the expectation of the difference above is replaced by its MC estimate in order to approximate it:
\begin{equation*}
    \lvert \mathbb{E}[Q_h - Q_{2h}] \rvert = \frac{1}{N} \left| \sum_{i=1}^N Q_h^{(i)} - Q_{2h}^{(i)} \right|,
\end{equation*}
where $N$ is chosen large enough to obtain a suitable estimate, and $Q_h^{(i)}$ and $Q_{2h}^{(i)}$ are computed with the same underlying sample of the random field $k(\mathbf{x}, \omega^{(i)})$. In the case of smoothing, $\tilde{Q}_{2h}^{(i)}$ is calculated with the smoothed sampled used in the approximation of $Q_h^{(i)}$. An analogous approach can be followed to obtain estimates for $\gamma$ and $\beta$. These are necessary for establishing the theoretical bound for the cost of the MLMC algorithm using Theorem \ref{thm: the-only-theorem}.

Further, throughout all our numerics, we use $h_\ell = 2^{-\ell} h_0$ as discretisation levels in the MLMC estimator. We use the same mesh size on level $\ell$ for generating the random field samples.

\begin{figure}
\centering
\begin{subfigure}[t]{.49\textwidth}
  \centering
  \includegraphics[width=\linewidth]{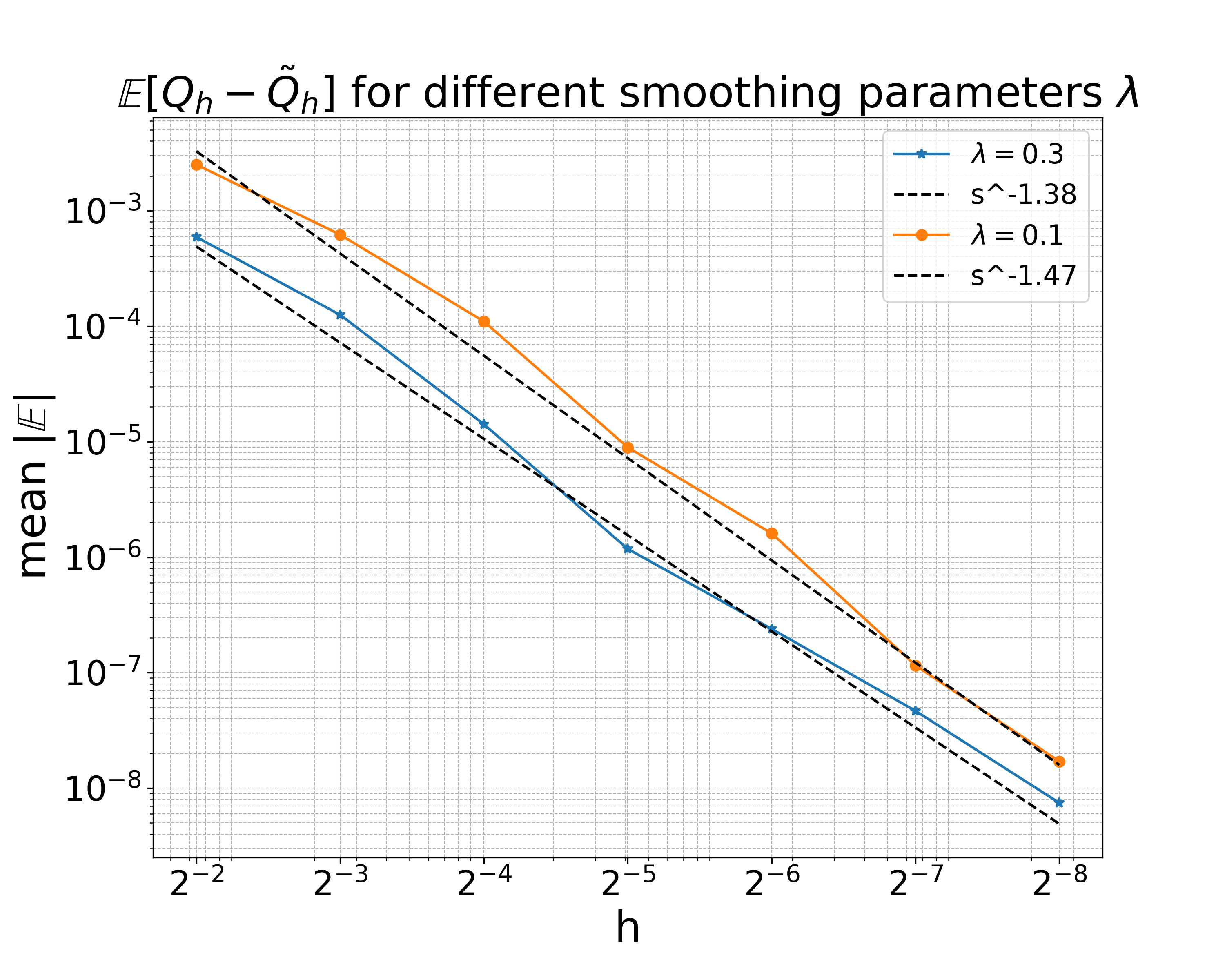}
  \caption{$\log \lvert \mathbb{E}[Q_{h} - \tilde{Q}_h] \rvert$.}
  \label{fig: mean-smoothing-error-expo}
\end{subfigure}%
\begin{subfigure}[t]{.49\textwidth}
  \centering
  \includegraphics[width=\linewidth]{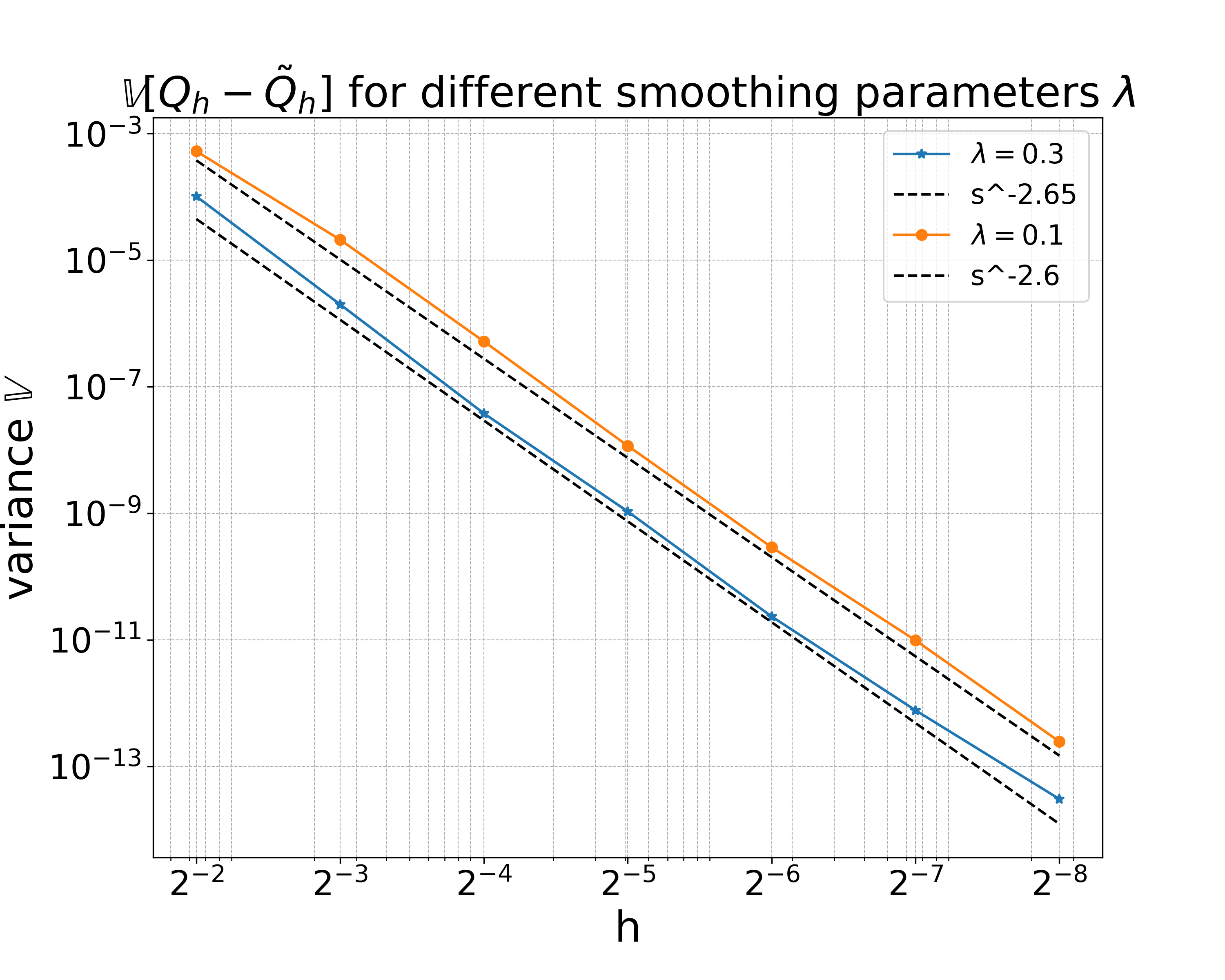}
  \caption{$\log \mathbb{V}[Q_{h} - \tilde{Q}_h]$.}
  \label{fig: variance-smoothing-expo}
\end{subfigure}
\caption{Decay rate of $\lvert \mathbb{E}[Q_{h} - \tilde{Q}_h] \rvert$ and $\mathbb{V}[Q_{h} - \tilde{Q}_h]$ with $h$ on a log scale, for the exponential covariance \eqref{eq: p-norm-cov-function} with $\lambda = 0.3$ and $\lambda = 0.1$, using ${\tau = \sqrt{s}}$. The quantity of interest is $Q = u(\mathbf{x}^*, \mathbf{y}^*)$.}
\label{fig: smoothing-error-expo}
\end{figure}

Finally, we use CPU times to quantify the computational cost of the studied algorithms. These were obtained by running our Python implementation remotely on the Cirrus High Performance Computing system {(2.1 GHz, 18-core Intel Xeon E5-2695 (Broadwell) series processor, 256 GB of memory \cite{epcc_cirrus_2023})}. In addition, we exploit Python's $\texttt{time}$ library to measure the CPU time, and quantify the error in terms of the RMSEs $e\left(\widehat{Q}^{\text{MC}}_{h,N}\right)$ and $e\left(\widehat{Q}^{\text{MLMC}}_{L}\right)$.

\begin{figure}
\centering
\begin{subfigure}[t]{.49\textwidth}
  \centering
  \includegraphics[width=\linewidth]{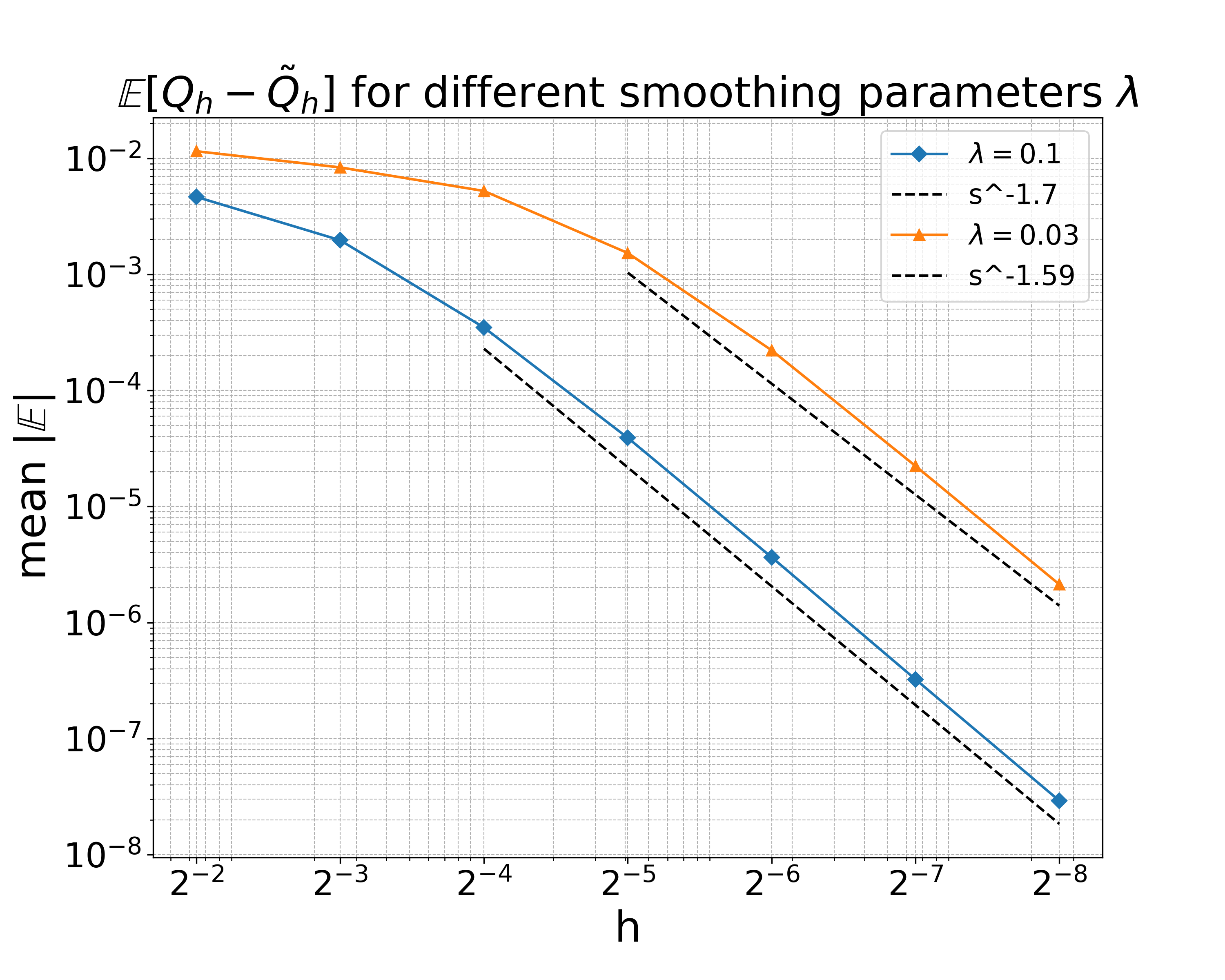}
  \caption{$\log \lvert \mathbb{E}[Q_{h} - \tilde{Q}_h] \rvert$.}
  \label{fig: mean-smoothing-error-mat}
\end{subfigure}%
\begin{subfigure}[t]{.49\textwidth}
  \centering
  \includegraphics[width=\linewidth]{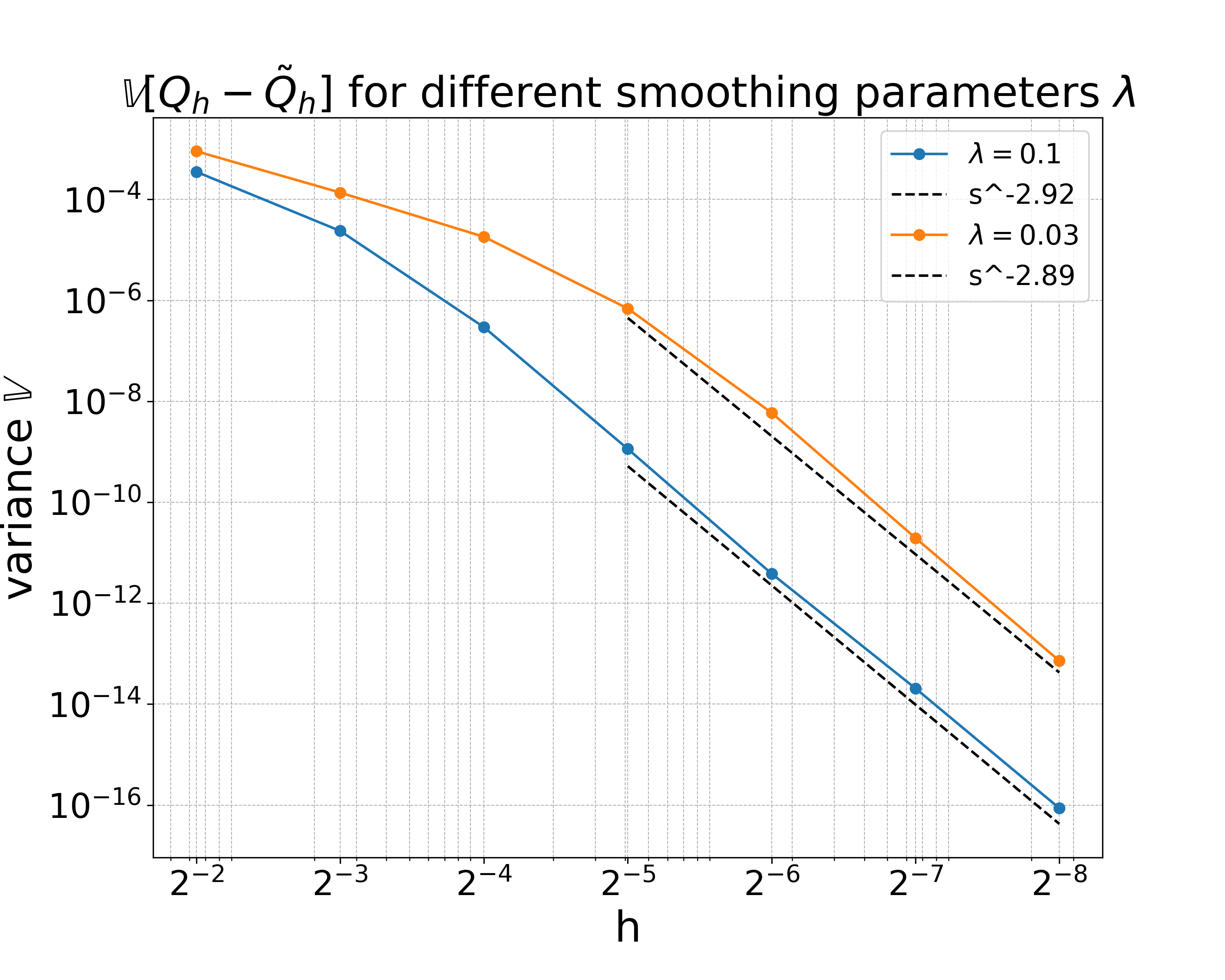}
  \caption{$\log \mathbb{V}[Q_{h} - \tilde{Q}_h]$.}
  \label{fig: variance-smoothing-mat}
\end{subfigure}
\caption{Decay rate of $\lvert \mathbb{E}[Q_{h} - \tilde{Q}_h] \rvert$ and $\mathbb{V}[Q_{h} - \tilde{Q}_h]$ with $h$ on a log scale, for the Mat\'ern covariance \eqref{eq: matern-cov-function} with $\lambda = 0.1$ and $\lambda = 0.03$. Here, {${\tau}$ is computed using Eq. \eqref{eq: tau-definition}}. The quantity of interest is $Q = \|u(\cdot, \omega)\|_{L^2(D)}$. }
\label{fig: smoothing-error-mat}
\end{figure}

\subsection{Smoothing error}
\label{sec: smoothing-error-numerics}

Here, we explore a numerical example aimed at demonstrating the theoretical bounds in Theorem \ref{thm: error-functional-sample}. We choose the exponential covariance \eqref{eq: p-norm-cov-function} with $p=1$, $\sigma^2 = 1$, and use {$\lambda = 0.3$ and $\lambda=0.1$}, so that we are in the setting of Corollary \ref{cor: error-1-norm-cov}. For the functional $Q = \|u(\cdot, \omega)\|_{L^2(D)}$, we expect $\alpha = 1$ from \citep{teckentrup_further_2013}. Hence, letting $\alpha = 1$ in Eq. \eqref{eq: k-values-p-norm-cov}, we obtain:
\begin{equation*}
    {\tau = \frac{s+1}{\frac{C_s}{2C_\alpha} s^{\frac{1}{2}}+ 1}}.
\end{equation*}
From Corollary \ref{cor: error-1-norm-cov}, this choice of ${\tau}$ yields the following bounds:
\begin{equation*}
    \mathbb{E}[|Q_{h} - \tilde{Q}_h|] \lesssim s^{-\frac{1}{2}} \quad \text{ and } \quad \mathbb{E}[|Q_{h} - \tilde{Q}_h|^2] \lesssim s^{-1}.
\end{equation*}


{Thus, in Figure \ref{fig: smoothing-error-expo}, we show the decline in $\lvert\mathbb{E}[Q_{h} - \tilde{Q}_h] \rvert$ and $\mathbb{V}[Q_{h} - \tilde{Q}_h]$ with $h$ on a log scale for $\tau = \sqrt{s}$. We observe that we obtain a much faster decay rate than predicted by the theory.}


{We repeat the above experiment for the Mat\'ern covariance with $\sigma^2 = 1$, $\nu = 1.5$ and use $\lambda = 0.1$ and $\lambda = 0.03$. We solve Eq. \eqref{eq: tau-values-matern} to find the values of $\tau$ which yield the same decay rate as in the separable exponential case above. The results are shown in Figure \ref{fig: smoothing-error-mat}, where we observe a similar trend to the one in Figure \ref{fig: smoothing-error-expo}.}


\subsection{Random field with exponential covariance}
\label{sec: 0.3-example}

In this section, we analyse the computational complexity of MC and MLMC using circulant embedding with and without smoothing. In particular, we focus on the exponential covariance function \eqref{eq: p-norm-cov-function} with norm $p=1$ and correlation lengths $\lambda = 0.3$ and $\lambda = 0.1$.

\begin{figure}
\centering
\begin{subfigure}[b]{.49\textwidth}
  \centering
  \includegraphics[width=\linewidth]{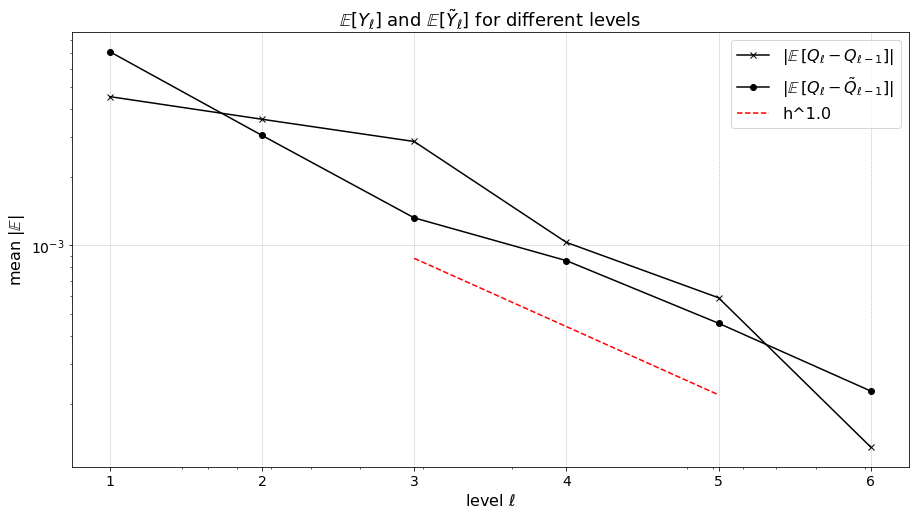}
  \caption{$\log \lvert \mathbb{E}[Q_{h_\ell} - Q_{h_{\ell-1}}] \rvert$ and $\log \lvert \mathbb{E}[Q_{h_\ell}-\tilde{Q}_{h_{\ell-1}}] \rvert$.}
  \label{fig: 0.3-alpha}
\end{subfigure}%
\begin{subfigure}[b]{.49\textwidth}
  \centering
  \includegraphics[width=\linewidth]{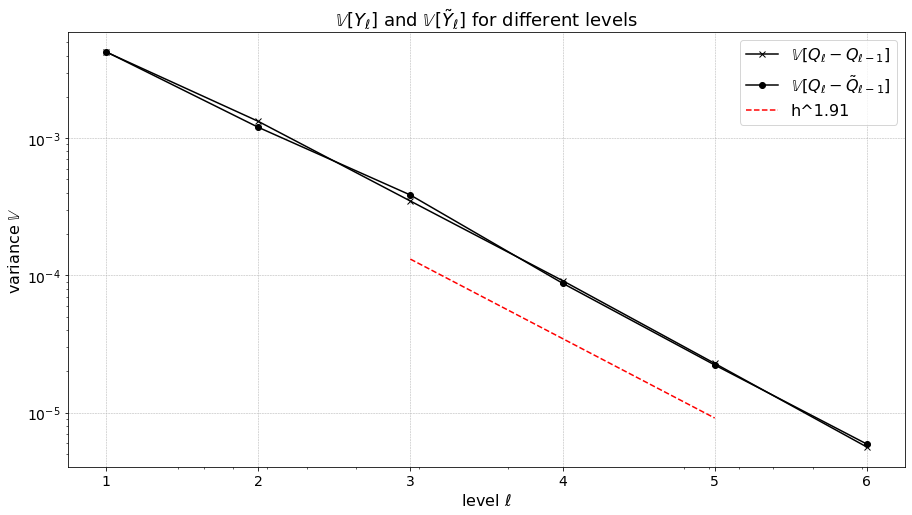}
  \caption{$\log \mathbb{V}[Q_{h_\ell} - Q_{h_{\ell-1}}]$ and $\log \mathbb{V}[Q_{h_\ell}-\tilde{Q}_{h_{\ell-1}}]$.}
  \label{fig: 0.3-beta}
\end{subfigure}

\begin{subfigure}[b]{.49\textwidth}
  \centering
  \includegraphics[width=\linewidth]{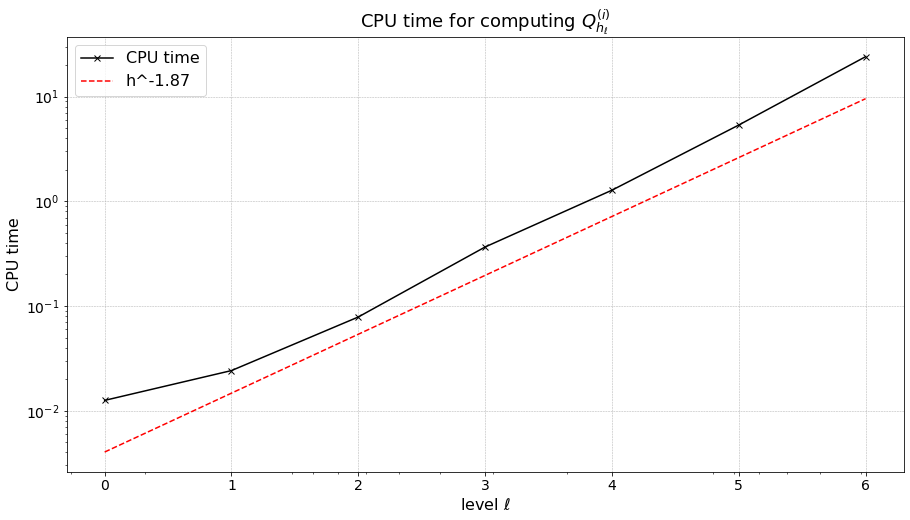}
  \caption{Cost per sample $\mathcal{C} (Q_{h_\ell}^{(i)})$.}
  \label{fig: 0.3-gamma}
\end{subfigure}
\begin{subfigure}[b]{.49\textwidth}
  \centering
  \includegraphics[width=\linewidth]{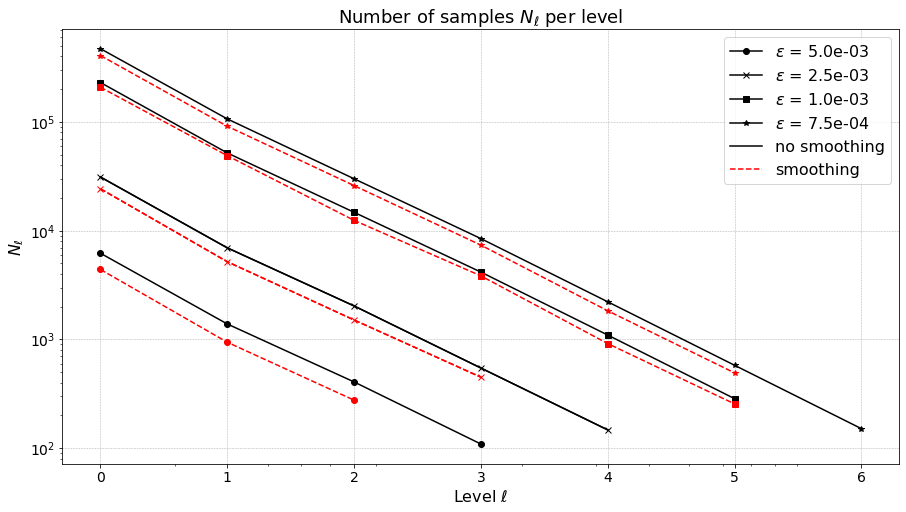}
  \caption{Number of samples $N_\ell$ for different $\varepsilon$.}
  \label{fig: 0.3-nl}
\end{subfigure}
\caption{Plots of $\alpha$ $\beta$, and $\gamma$ and their best linear fit, and of the number of samples $N_\ell$ on a log scale computed with and without smoothing for $\lambda = 0.3$ and $\sigma^2 = 1$ using MLMC with $\ell = 0, \dots, 6$. The quantity of interest is the mean value of the pressure at $\mathbf{x}^*=\left(\frac{7}{15},\frac{7}{15}\right)$. }
\label{fig: 0.3-alpha-beta}
\end{figure}

Accordingly, Figure \ref{fig: 0.3-alpha-beta} illustrates the values obtained for $\alpha$ and $\beta$ for the MC and MLMC algorithms used to compute the quantity of interest $\mathbb{E}[Q]$. In this case, we model the random coefficient in Eq. \eqref{eq: pde-model} as a random field whose covariance function has correlation length $\lambda = 0.3$ and variance $\sigma = 1$. As noted in \eqref{eq: mlmc-h0-bound-exp-covariance}, the optimal choice for the coarsest MLMC level is such that $h_0$ is slightly smaller than $\lambda$, which yields $h_0 = 2^{-2}$. 

Specifically, Figure \ref{fig: 0.3-alpha} shows the decay in $\log |\mathbb{E}[Q_h - Q_{2h}]|$, which, due to the definition of the MLMC levels, is equal to $\log |\mathbb{E}[Q_{h_\ell} - Q_{h_{\ell-1}}]|$, for $\ell = 1, \dotsc, 6$. The line through these points has a slope of around $1$, which implies $\alpha = 1$. Further, we use ${\tau_\ell = \sqrt{s_\ell}}$ for smoothing the samples on level $\ell$, and the same coarsest mesh $h_0 = 2^{-2}$ for MLMC with smoothed random field samples. As expected, this gives the same slope $\alpha$, with a slightly smaller constant $\tilde{C}_\alpha$.

In addition, Figure \ref{fig: 0.3-beta} portrays the decrease in $\log \mathbb{V}[Y_\ell]$ with level $\ell$, $\ell = 1, \dotsc, 6$. Following the same procedure as above, this gives $\beta \approx 1.91$, which confirms that $\beta \approx 2 \alpha$, as expected from the theory. Similar to the $\alpha$ plot, if we introduce smoothing in the MLMC algorithm with ${\tau_\ell = \sqrt{s_\ell}}$, the $\mathbb{V}[Y_\ell]$ values are unaltered, so that $\beta$ remains unchanged. 

Further, Figure \ref{fig: 0.3-gamma} depicts the growth in the cost of computing one sample $Q_{h_\ell}^{(i)}$ with level $\ell$. On a log scale, the resulting line has a slope of approximately $1.87$, so that $\gamma \approx 1.87$. This does not change when smoothing is introduced for the following reasons:
\begin{itemize}
    \item on the one hand, the main contribution to the cost arises from numerically solving the PDE \eqref{eq: pde-model} which is invariable to the type of samples involved;
    \item the complexity of the $\texttt{FFT}$ routine of $\mathcal{O}(s \log s)$ does not depend on the entries in the vector of eigenvalues $\boldsymbol{\Lambda} \in \mathbb{R}^s$, only on its dimension.
\end{itemize}

Moreover, Figure \ref{fig: 0.3-nl} presents the number of samples $N_\ell$ required per level $\ell = 0, \dotsc, L$ for different accuracies $\varepsilon$, where $N_\ell$ is computed using Eq. \eqref{eq: optimal-nl}. This is shown for the MLMC estimator where the random field samples are smoothed with ${\tau_\ell = \sqrt{s_\ell}}$, as before. For the imposed RMSE, we adopt equally spaced values between $\varepsilon = 10^{-1}$ and $\varepsilon = 5 \cdot 10^{-4}$. This figure also emphasises the discrepancy between the number of samples needed on the coarsest level $h_0$ and the finest level $h_L$. For example, to achieve an accuracy of $10^{-3}$, five levels are necessary for the MLMC estimator which integrates smoothing technique. Particularly, $N_0 \approx 10^5$ samples are computed on the coarsest level, while $N_5 \approx 10^2$ samples are estimated on the finest level.

\begin{figure}[t]
\centering
\includegraphics[width=0.49\textwidth]{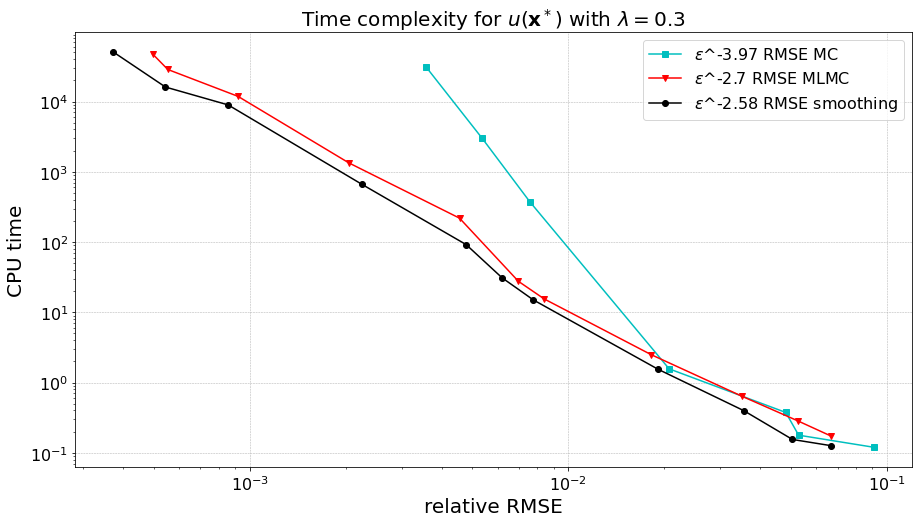}
\caption{Performance plot on a log scale for $\lambda = 0.3$ and $\sigma^2 = 1$ using Monte Carlo and Multilevel Monte Carlo, with and without smoothing. The quantity of interest is the mean value of the pressure at $\mathbf{x}^*=\left(\frac{7}{15},\frac{7}{15}\right)$.}
\label{fig: 0.3-mc-mlmc-smoothing}
\end{figure}

Lastly, we expect $\mathcal{C}_\varepsilon\left(\widehat{Q}^{\text{MC}}_{h, N}\right) \lesssim \varepsilon^{-2-\gamma/\alpha} \approx \varepsilon^{-3.87}$ for the MC estimator, as explained at the end of Section \ref{sec: mc-simulations}. We also anticipate from Theorem \ref{thm: the-only-theorem} that $\mathcal{C}_\varepsilon\left(\widehat{Q}^{\text{MLMC}}_L\right) \lesssim \varepsilon^{-2} (\log \varepsilon)^2$ for the MLMC estimator both with and without smoothing, as $\beta \approx \gamma$ in both cases. Accordingly, Figure \ref{fig: 0.3-mc-mlmc-smoothing} displays, on a log scale, the CPU time taken to achieve a relative RMSE less than $\varepsilon$ with each of the MC, MLMC with and without smoothing estimators. For the MLMC simulations we use the same values of $\varepsilon$ as in the $N_\ell$ figure, while for the MC simulations we only consider equally spaced values between $\varepsilon = 10^{-1}$ and $\varepsilon = 5 \cdot 10^{-3}$, due to the long CPU times involved. 

Indeed, this graph illustrates that the best performing algorithm is the MLMC with smoothing. In particular, we observe that the MC cost has a slope of roughly $-3.97$, as expected, while both versions of the MLMC estimator achieve a computational complexity of $C_\varepsilon \varepsilon^{-2.6}$. The main contribution that the smoothing technique brings is to decrease the constant $C_\varepsilon$. Specifically, if no smoothing is introduced, $C_\varepsilon \approx 7.5 \cdot 10^{-5}$, while with smoothing we have $C_\varepsilon \approx 5.5 \cdot 10^{-5}$.  

\begin{figure}[t]
\begin{subfigure}[t]{.49\textwidth}
  \centering
  \includegraphics[width=\linewidth]{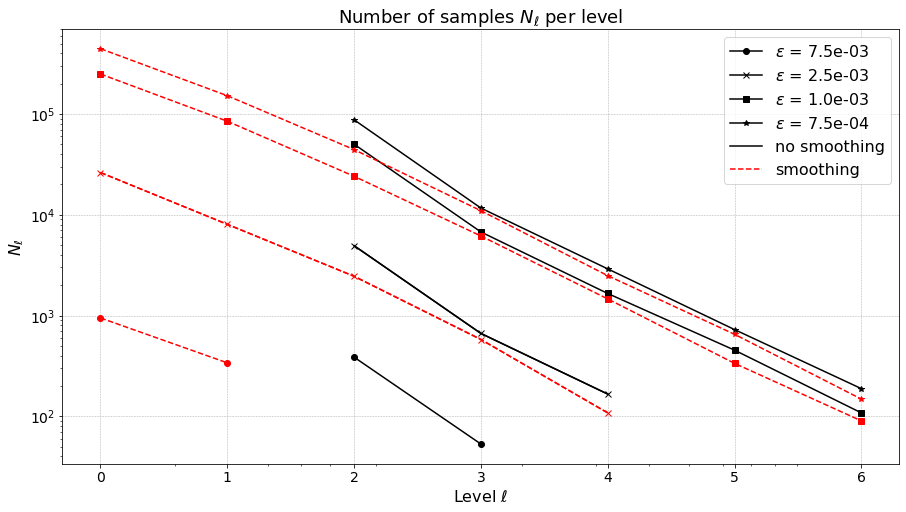}
  \caption{Number of samples $N_\ell$ for different $\varepsilon$.}
  \label{fig: 0.1-nl}
\end{subfigure}
\begin{subfigure}[t]{.49\textwidth}
  \centering
  \includegraphics[width=\linewidth]{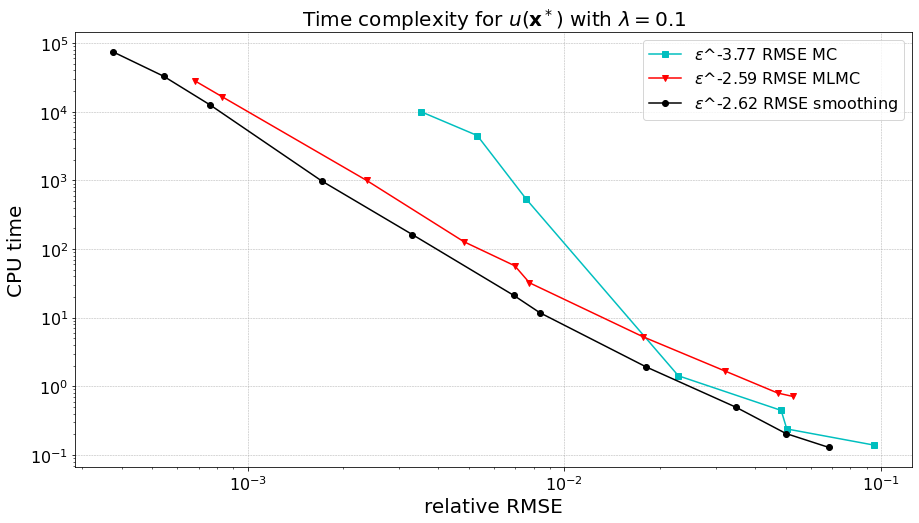}
  \caption{CPU times of MC, MLMC, MLMC-CES.}
  \label{fig: 0.1-mc-mlmc-smoothing}
\end{subfigure}
\caption{Plot of the number of samples $N_\ell$ on a log scale and relative RMSE for $\lambda = 0.1$ and $\sigma^2 = 1$ using Monte Carlo and Multilevel Monte Carlo with and without smoothing with $\ell = 0, \dots, 6$. The quantity of interest is the mean value of the pressure at $\mathbf{x}^*=\left(\frac{7}{15},\frac{7}{15}\right)$.}
\label{fig: 0.1-nl-rmse}
\end{figure}

Note that integrating the smoothing technique in the MLMC algorithm does not yield major improvements in the computational cost in this case. This is expected {since $\lambda$ is still relatively large and the restriction $h_0 \leq 2^{-2}$ is not too severe.} We use $h_0 = 2^{-2}$ for the coarsest grid in both the smoothing and no-smoothing examples. 

Let us now investigate the case when the covariance function of $k(\mathbf{x}, \omega)$ has correlation length $\lambda = 0.1$. Here, the MLMC method {without smoothing} is severely limited in the choice of coarsest mesh, namely $h_0 = 2^{-4}$. In this case, smoothing the random field samples plays an essential role, since it allows $h_0 = 2^{-2}$ as before. {Note that the benefit of the smoothing decreases as $\varepsilon$ decreases. This is expected, since a very fine mesh size $h_L$ is required for small $\varepsilon$ and the restriction on the coarsest mesh hence poses less severe restrictions on the number of levels that can be used in the MLMC estimator.}

In this case, the plots for $\alpha$, $\beta$, and $\gamma$ look analogous, as only the constants $C_\alpha$, $C_\beta$ and $C_\gamma$ are different, and so we do not present these here. Rather, we focus on how the number of samples per level $N_\ell$ varies, and how the overall computational complexities of MC and MLMC with and without smoothing change.

Specifically, Figure \ref{fig: 0.1-nl} illustrates the number of samples $N_\ell$ per level for different accuracies $\varepsilon$. It is important to highlight that, in this case, level $0$ for MLMC without smoothing has a mesh size $h_0 = 2^{-4}$, which corresponds to level $2$ for the MLMC with smoothing estimator. Hence, we observe that, for an accuracy of $\varepsilon = 10^{-3}$ for instance, standard MLMC uses five levels, while MLMC-CES uses seven levels. However, in the former, the mesh sizes vary from $h_0 = 2^{-4}$ to $h_4 = 2^{-8}$, while in the latter the grid resolutions range from $h_0 = 2^{-2}$ to $h_5 = 2^{-7}$. This yields significant computational savings, as it allows for cheap approximations to be exploited.

The argument above is supported by Figure \ref{fig: 0.1-mc-mlmc-smoothing}, which displays the computational complexity of the three surveyed algorithms. In particular, we expect $\mathcal{C}_\varepsilon\left(\widehat{Q}^{\text{MC}}_{h, N}\right) \lesssim \varepsilon^{-3.87}$ and $\mathcal{C}_\varepsilon\left(\widehat{Q}^{\text{MLMC}}_L\right) = C_\varepsilon \varepsilon^{-2} (\log \varepsilon)^2$ for the MLMC estimator both with and without smoothing, as $\alpha$, $\beta$ and $\gamma$ are as in the $\lambda = 0.3$ case. Indeed, these are very closely recovered in our experiments, as highlighted in this figure. However, one crucial aspect here is the fact that the MC estimator outperforms MLMC for large accuracies, behaviour alleviated by introducing the smoothing. In addition, the constant $C_\varepsilon$ decreases from roughly $1.6 \cdot 10^{-4}$ in the no smoothing case to around $6.5 \cdot 10^{-5}$ for the estimator with smoothing, giving rise to the computational savings observed.

\subsection{Random field with Mat\'ern covariance}

Throughout this section, we concentrate on the Mat\'ern covariance in Eq. \eqref{eq: matern-cov-function}. We fix the variance $\sigma^2 = 1$ and the smoothness parameter $\nu = 1.5$, and vary the length scale $\lambda$. In particular, we explore the computational complexity of MC, MLMC and MLMC-CES for estimating the quantity of interest $Q = \|u(\cdot, \omega)\|_{L^2(D)}$ using $\lambda = 0.1$ and $\lambda = 0.03$. {In addition, we choose the values of $\tau$ by solving Eq. \eqref{eq: tau-values-matern} as in Section \ref{sec: smoothing-error-numerics}.}

\begin{figure}
\centering
\begin{subfigure}[b]{.49\textwidth}
  \centering
  \includegraphics[width=\linewidth]{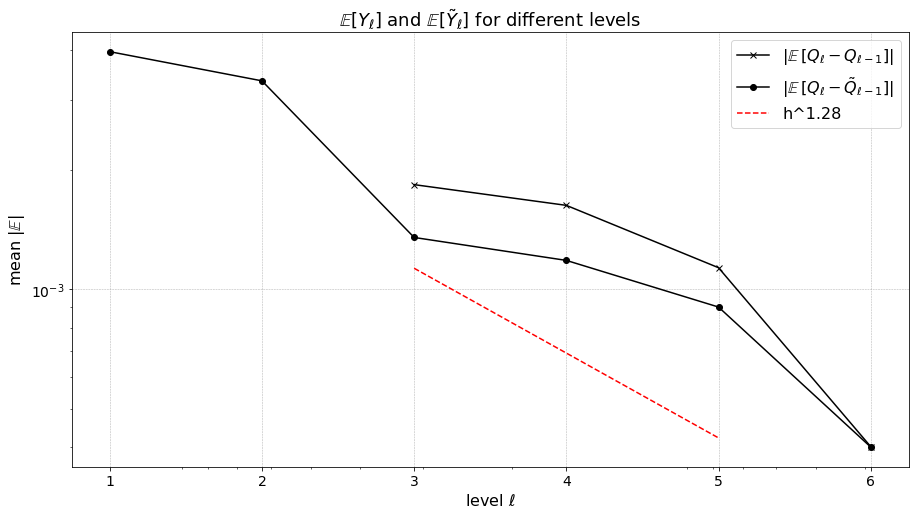}
  \caption{$\log \lvert \mathbb{E}[Q_{h_\ell} - Q_{h_{\ell-1}}] \rvert$ and $\log \lvert \mathbb{E}[Q_{h_\ell}-\tilde{Q}_{\ell-1}] \rvert$.}
  \label{fig: 0.03-alpha}
\end{subfigure}%
\begin{subfigure}[b]{.49\textwidth}
  \centering
  \includegraphics[width=\linewidth]{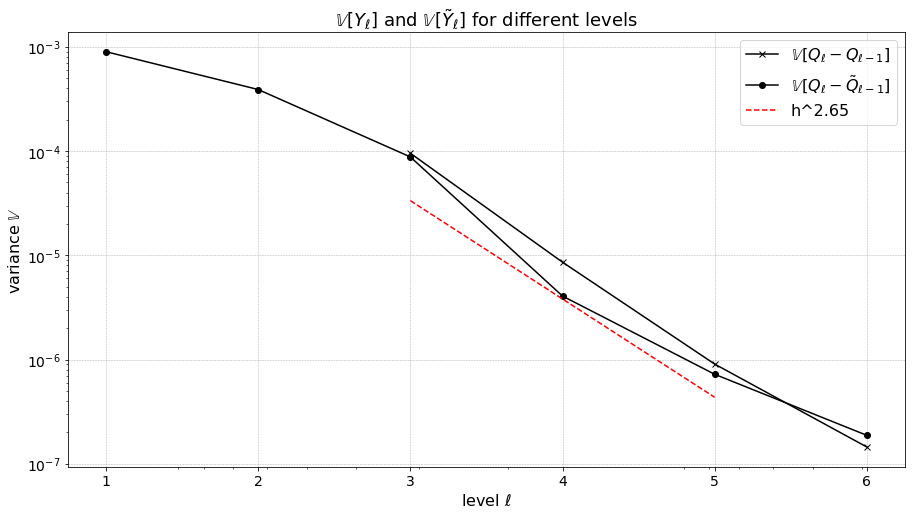}
  \caption{$\log \mathbb{V}[Q_{h_\ell} - Q_{h_{\ell-1}}]$ and $\log \mathbb{V}[Q_{h_\ell}-\tilde{Q}_{\ell-1}]$.}
  \label{fig: 0.03-beta}
\end{subfigure}

\caption{Plots of $\alpha$ and $\beta$ on a log scale for $\lambda = 0.03$, $\sigma^2 = 1$ and $\nu=1.5$ using Multilevel Monte Carlo with and without smoothing with $\ell = 0, \dots, 6$. The quantity of interest is the mean value of the L2 norm $\|u(\cdot, \omega)\|_{L^2(D)}$. }
\label{fig: 0.03-alpha-beta}
\end{figure}

Note that changing the covariance function does not alter the computational cost of computing one sample $\mathcal{C}(Q_{h_\ell}^{(i)})$, and so we keep $\gamma \approx 1.87$ throughout this section.

Now, Figure \ref{fig: 0.03-alpha-beta} portrays the values obtained for $\alpha$ and $\beta$ for the MC and MLMC algorithms used to compute the quantity of interest $\mathbb{E}[Q]$ when the random field coefficient in \eqref{eq: pde-model} has length scale $\lambda = 0.03$. Since the two cases $\lambda = 0.1$ and $\lambda = 0.03$ give similar results for $\alpha$ and $\beta$, we only focus on the latter here, similar to the previous section. In this case, Eq. \eqref{eq: mlmc-h0-bound-matern-covariance} suggests $h_0 = 2^{-4}$, which is the coarsest mesh we use in estimating $\alpha$ and $\beta$ without smoothing. 

In particular, Figure \ref{fig: 0.03-alpha} illustrates the mean order of convergence $\alpha$. This is achieved by computing $\log |\mathbb{E}[Q_{h_\ell} - Q_{h_{\ell-1}}]|$ and $\log |\mathbb{E}[Q_{h_\ell} - \tilde{Q}_{h_{\ell-1}}]|$ on different levels. Both lines have a slope of $\alpha \approx 1.28$ on average, but with two different constants $C_\alpha$ and $\tilde{C}_\alpha$.

Moreover, Figure \ref{fig: 0.03-beta} depicts the decay in $\log \mathbb{V}[Q_{h_\ell} - Q_{h_{\ell-1}}]$ and $\log \mathbb{V}[Q_{h_\ell} - \tilde{Q}_{\ell-1}]$ with level $\ell$. In particular, this graph indicates the MLMC estimators with and without smoothing have approximately the same $\beta \approx 2.6$, so that we have, again, $\beta \approx 2\alpha$. 

\begin{figure}[t]
\begin{subfigure}[t]{.49\textwidth}
  \centering
  \includegraphics[width=\linewidth]{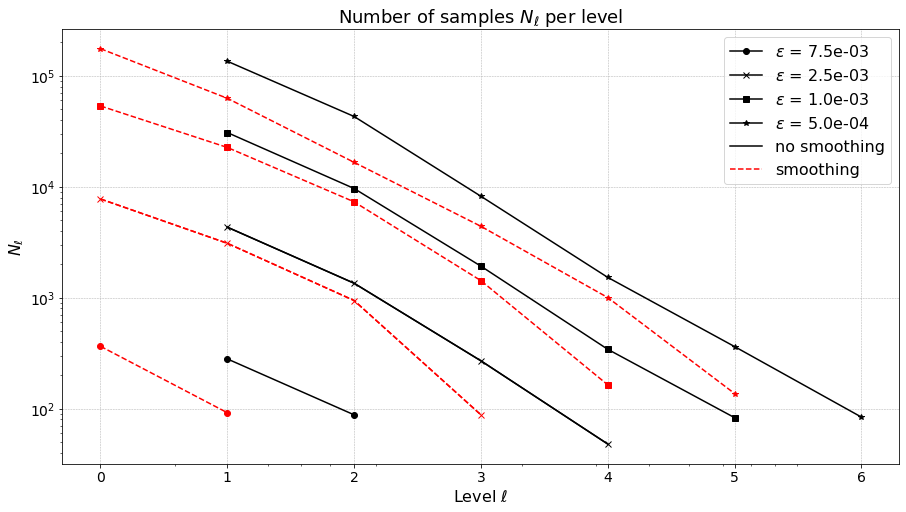}
  \caption{Number of samples $N_\ell$ for different $\varepsilon$.}
  \label{fig: 0.1-1.5-nl}
\end{subfigure}
\begin{subfigure}[t]{.49\textwidth}
  \centering
  \includegraphics[width=\linewidth]{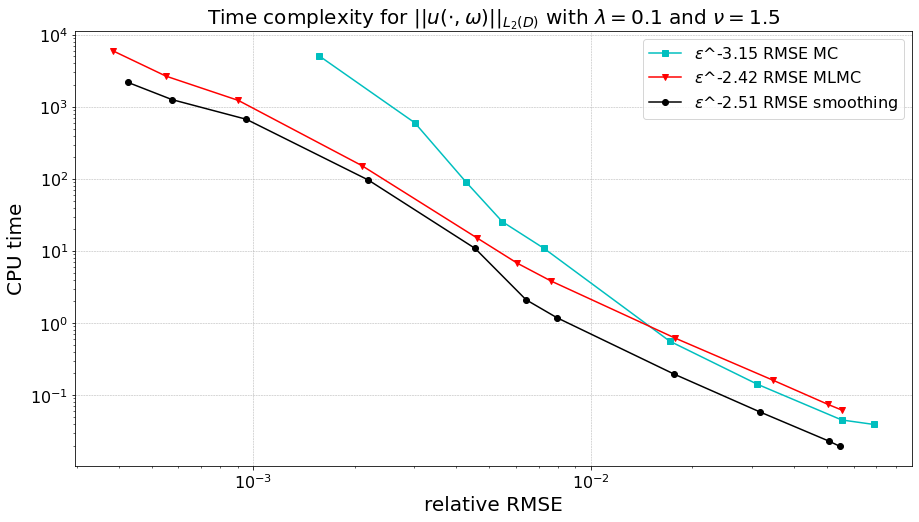}
  \caption{CPU times of MC, MLMC, MLMC-CES.}
  \label{fig: 0.1-1.5-mc-mlmc-smoothing}
\end{subfigure}
\caption{Plot of the number of samples $N_\ell$ on a log scale and relative RMSE for $\lambda = 0.1$, $\sigma^2 = 1$ and $\nu=1.5$ using Monte Carlo and Multilevel Monte Carlo with and without smoothing with $\ell = 0, \dots, 6$. The quantity of interest is the mean value of the L2 norm $\|u(\cdot, \omega)\|_{L^2(D)}$.}
\label{fig: 0.1-1.5-nl-rmse}
\end{figure}

In addition, Figure \ref{fig: 0.1-1.5-nl} portrays the number of samples $N_\ell$ needed per level $\ell = 0, \dotsc, L$ for various accuracies $\varepsilon$. This figure demonstrates that the drop from $N_L$ to $N_0$ is significant, as in the previous example. For instance, to achieve the smallest error considered here, namely $\varepsilon = 5 \cdot 10^{-4}$, roughly $10^{5}$ samples are necessary on the coarsest level of the MLMC estimator with smoothing with $h_0 = 2^{-1}$. Likewise, around $2 \cdot 10^2$ estimates are computed on level five, where $h_5 = 2^{-5}$. The dip is similar for the number of samples in the standard MLMC estimator, even though we use $h_0 = 2^{-2}$ in this case.

Figure \ref{fig: 0.1-1.5-mc-mlmc-smoothing} displays, on a log scale, the computational efficiency of the three surveyed algorithms to achieve a relative RMSE less than $\varepsilon$, for given $\varepsilon$. We anticipate $\mathcal{C}_\varepsilon\left(\widehat{Q}^{\text{MC}}_{h, N}\right) \lesssim \varepsilon^{-2-\gamma/\alpha} \approx \varepsilon^{-3.32}$ for the MC estimator, which is comparable to $\varepsilon^{-3.15}$ observed in our numerics. While expect from Theorem \ref{thm: the-only-theorem} that $\mathcal{C}_\varepsilon\left(\widehat{Q}^{\text{MLMC}}_L\right) \lesssim \varepsilon^{-2}$ as $\gamma < \beta$ for the MLMC estimators with and without smoothing, we still observe $\mathcal{C}_\varepsilon\left(\widehat{Q}^{\text{MLMC}}_L\right) \lesssim \varepsilon^{-2} (\log \varepsilon)^2$. This could be due to the initial drop in the $\mathbb{V}[Y_\ell]$, which is significantly faster than the decay corresponding to the later levels. Specifically, if we use the last three levels to estimate $\beta$, we obtain, in fact, $\beta \approx 2$.

Certainly, this plot establishes, again, that the most efficient method is the MLMC estimator with integrated smoothing. First, the regression line through the points $\left(\varepsilon, \mathcal{C}_\varepsilon \left(\widehat{Q}^{\text{MC}}_{h,N} \right) \right)$ has, on a log scale, a slope of $-3.15$. Second, both versions of the MLMC approach achieve a cost of roughly $C_\varepsilon \varepsilon^{-2.5}$. The difference rests on the value of $C_\varepsilon$. Specifically, if no smoothing is used, we have $C_\varepsilon \approx 4 \cdot 10^{-5}$, while with smoothing we achieve $C_\varepsilon \approx 10^{-5}$. 

Similar to the $\lambda = 0.1$ case in the above section, the computational savings prompted by the smoothing technique are not significant. This is because the coarsest mesh we can use in the standard MLMC estimator does not yield meaningful limitations in terms of number of levels we can exploit. As already mentioned, if we use $\nu = 1.5$ and $\lambda = 0.03$ in the Mat\'ern covariance, we obtain $h_0 = 2^{-4}$, in which case we expect considerable savings in computational cost.

\begin{figure}[t]
\begin{subfigure}[t]{.49\textwidth}
  \centering
  \includegraphics[width=\linewidth]{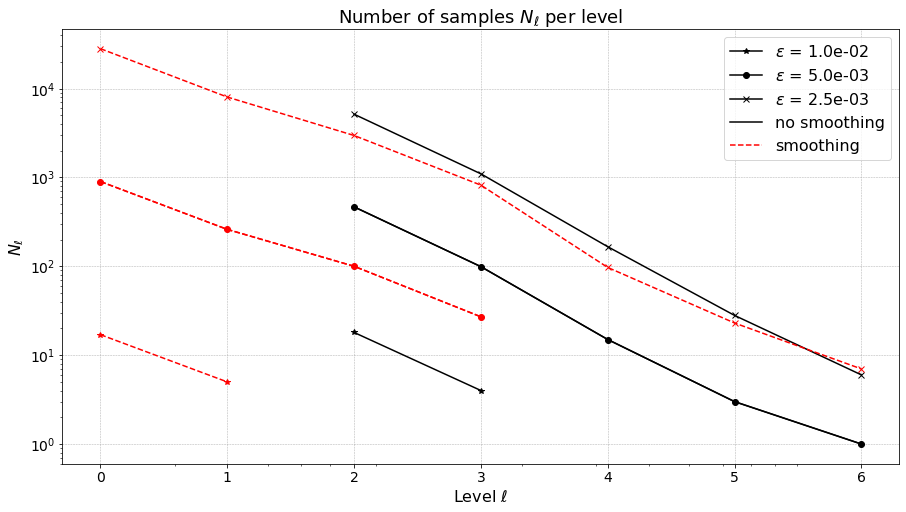}
  \caption{Number of samples $N_\ell$ for different $\varepsilon$.}
  \label{fig: 0.03-nl}
\end{subfigure}
\begin{subfigure}[t]{.49\textwidth}
  \centering
  \includegraphics[width=\linewidth]{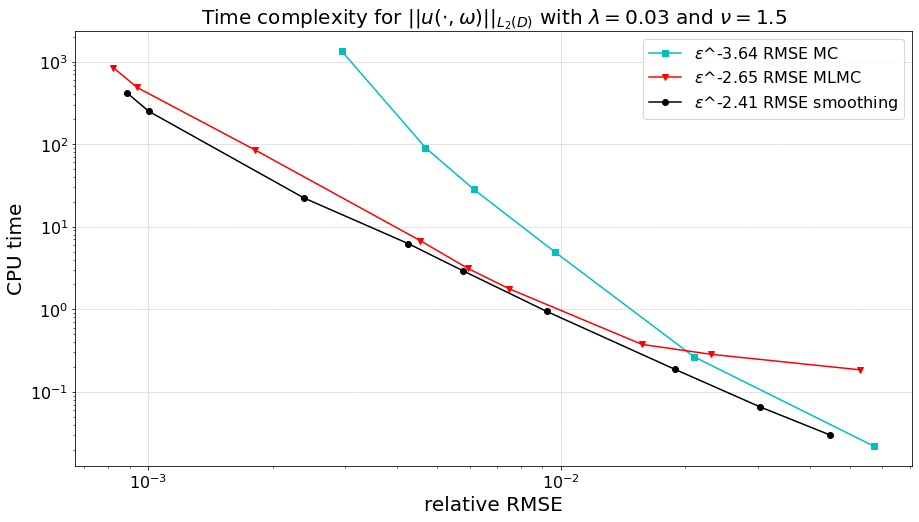}
  \caption{CPU times of MC, MLMC, MLMC-CES.}
  \label{fig: 0.03-mc-mlmc-smoothing}
\end{subfigure}
\caption{Plot of the number of samples $N_\ell$ on a log scale and relative RMSE for $\lambda = 0.03$, $\sigma^2 = 1$ and $\nu = 1.5$ using Monte Carlo and Multilevel Monte Carlo with and without smoothing with $\ell = 0, \dots, 6$. The quantity of interest is the mean value of the L2 norm $\|u(\cdot, \omega)\|_{L^2(D)}$.}
\label{fig: 0.03-nl-rmse}
\end{figure}

To support this, Figure \ref{fig: 0.03-nl} illustrates the number of samples $N_\ell$ in the case when $\lambda = 0.3$. Since level $0$ for MLMC without smoothing has a mesh size $h_0 = 2^{-4}$, this corresponds to level $2$ for the MLMC-CES estimator. While we observe that $N_0 \approx 2 \cdot 10^{4}$ are necessary for an accuracy of $2.5 \cdot 10^{-3}$ when smoothing is employed, these samples are computed on a grid with resolution $2^{-2}$, which are extremely cheap to estimate. In comparison, less than 10 samples are needed on the finest level. On the other hand, around $5 \cdot 10^3$ samples are sufficient for standard MLMC, but these are computed on a mesh of size $2^{-4}$.

This discussion is precisely echoed in Figure \ref{fig: 0.03-mc-mlmc-smoothing}, which showcases the computational complexity of MC, MLMC and MLMC-CES. In particular, note that the estimator with smoothing outperforms the other two algorithms. While it is substantially better than MC, whose cost satisfies $\mathcal{C}_\varepsilon\left(\widehat{Q}^{\text{MC}}_{h, N}\right) \lesssim \varepsilon^{-3.64}$, the main gain over the MLMC cost hinges on the constant $C_\varepsilon$ in $\mathcal{C}_\varepsilon\left(\widehat{Q}^{\text{MLMC}}_L\right) = C_\varepsilon \varepsilon^{-2}(\log \varepsilon)^2$. Specifically, for standard MLMC we recover $C_\varepsilon \approx 10^{-5}$, while for MLMC with smoothing we estimate $C_\varepsilon \approx 5 \cdot 10^{-6}$.

This feature is especially visible for larger accuracies $\varepsilon$. In fact, note that for $\varepsilon \geq 2.5 \cdot 10^{-2}$, the MLMC estimator with no smoothing is more costly to compute than the standard MC one. On the other hand, if we smooth the random field samples, the resulting procedure is faster than both. In particular, it is 10 times faster than the standard MLMC method for large accuracies. As we decrease the accuracy and more levels are included in the approximation, the performance of the two MLMC methods becomes comparable. Specifically, for $\varepsilon = 10^{-3}$,  MLMC-CES is roughly twice as fast than the standard MLMC.

Certainly, the benefits prompted by integrating the smoothing technique in the MLMC algorithm would be more compelling for even smaller correlation lengths. Guided by Eq. \eqref{eq: mlmc-h0-bound-exp-covariance} and \eqref{eq: mlmc-h0-bound-matern-covariance}, for $\nu = 0.5$ and $\lambda = 0.01$, we require $h_0 = 2^{-7}$ in the exponential covariance case and $h_0 = 2^{-6}$ in the Mat\'ern case, which limits severely the subsequent grid resolutions we can use. Hence, the circulant embedding with smoothing method would be essential in these cases.

\subsection{Comparison with KL-expansion}

\begin{figure}
\centering
\begin{subfigure}[t]{.49\textwidth}
  \centering
  \includegraphics[width=\linewidth]{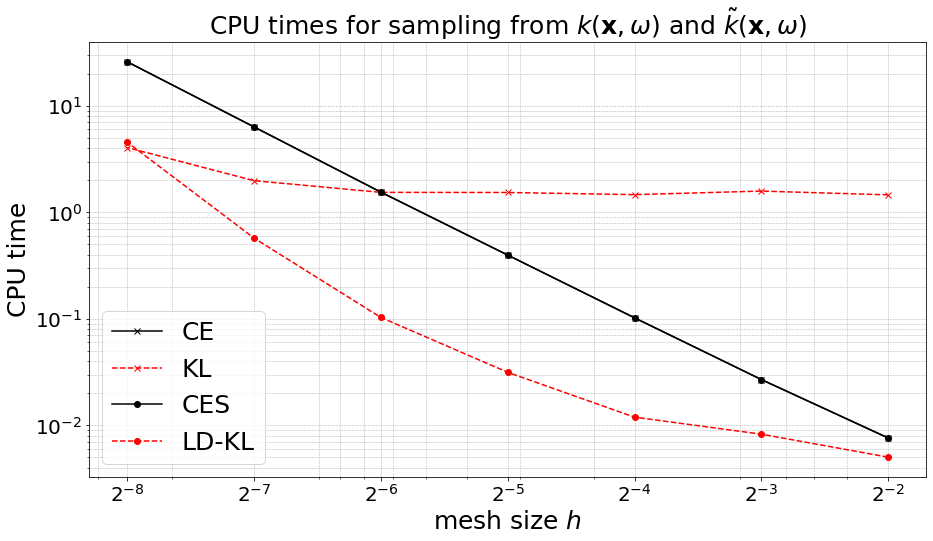}
  \caption{Cost per sample $k(\cdot, \omega)$ and $\tilde{k}(\cdot, \omega)$.}
  \label{fig: ld-ces-per-sample}
\end{subfigure}%
\begin{subfigure}[t]{.49\textwidth}
  \centering
  \includegraphics[width=\linewidth]{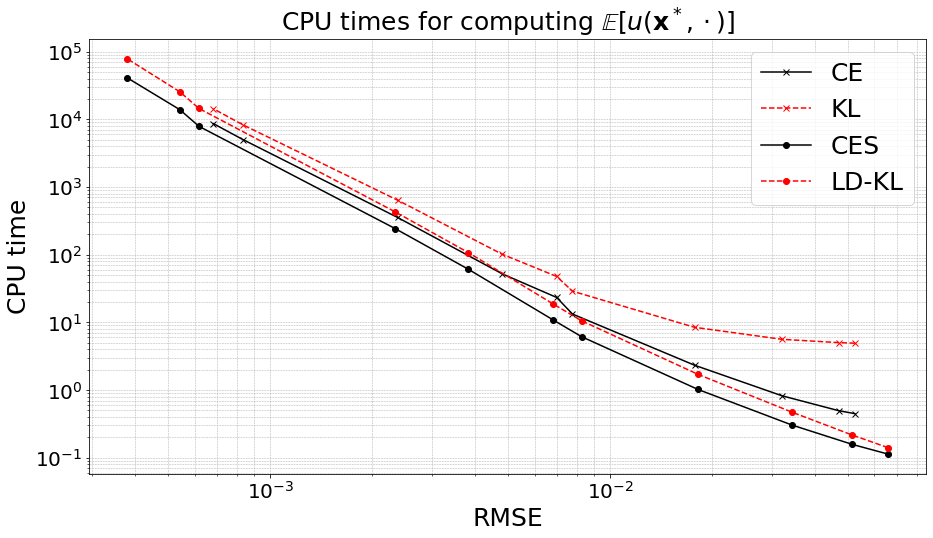}
  \caption{MLMC cost.}
  \label{fig: ld-ces-mlmc}
\end{subfigure}
\caption{Comparison plots for circulant embedding with and without smoothing, uniform and level-dependent truncation of KL-expansion using covariance \eqref{eq: p-norm-cov-function} with parameters $\lambda = 0.1$ and $\sigma^2 = 1$. The quantity of interest is the mean value of the pressure at $\mathbf{x}^*=\left(\frac{7}{15},\frac{7}{15}\right)$.}
\label{fig: ld-ces-cost}
\end{figure}

An alternative solution to the problem at hand is the level-dependent truncation of the KL-expansion of the random field $k(\mathbf{x}, \omega)$, proposed in \citep{teckentrup_further_2013, gittelson_multi-level_2013, schwab_multilevel_2023}. The aim of this section is to provide a brief comparison of the computational cost of this method and the circulant embedding with smoothing approach in the context of MLMC applied to the PDE \eqref{eq: pde-model}. 

To this end, we focus on the exponential covariance in Eq. \eqref{eq: p-norm-cov-function} with norm $p=1$, variance $\sigma^2 = 1$, and correlation length $\lambda = 0.1$, and consider the cost of estimating the mean value of the point evaluation $u(\mathbf{x}^*, \cdot)$, for $\mathbf{x}^* = \left(\frac{7}{15}, \frac{7}{15}\right)$.

Thus, Figure \ref{fig: ld-ces-cost} offers performance plots for the circulant embedding method with and without smoothing and the KL-expansion with uniform and level-dependent truncation. Specifically, Figure \ref{fig: ld-ces-per-sample} illustrates the computational cost achieved for computing one sample of $k(\mathbf{x}, \omega)$ and $\tilde{k}(\mathbf{x}, \omega)$, respectively, for different grid resolutions $h \in \{2^{-2}, \dotsc, 2^{-8}\}$. As explained in Section \ref{sec: 0.3-example}, there is practically no difference in the CE and CES costs. Further, for the KL-expansion, we use $m_{\text{KL}} = 2500$ modes, and for the level-dependent truncation we use the rule in \citep{teckentrup_further_2013} to choose the number of modes to include on each mesh. In the first case, the main contribution to the cost stems from computing the eigenvalues and eigenfunctions corresponding to the expansion \eqref{eq: kl-expansion}, so that the discretisation parameter $h$ is negligible. This explains the plateau up to $h = 2^{-7}$, where the small mesh size induces a more significant cost. In the case of the level-dependent truncation, as the number of modes is considerably smaller on each grid, the computational cost also scales accordingly. In particular, we observe that for one realisation $\tilde{k}(\mathbf{x}, \omega)$, the level-dependent truncation is more efficient than the smoothed circulant embedding.

On the other hand, Figure \ref{fig: ld-ces-mlmc} demonstrates that, once we sample repeatedly from the random field and also take into account the cost of solving the PDE, the MLMC with circulant embedding and its smoothing variant are more cost-effective than MLMC with the uniform and level-dependent truncation of the KL-expansion by roughly a factor of two. This is because, in order to integrate the KL-expansion in the finite element method, in addition to computing the eigenvalues and eigenfunctions in \eqref{eq: kl-expansion}, for each PDE solve we must perform a matrix-vector multiplication of the form $A\boldsymbol{\xi}$, where $A \in \mathbb{R}^{n \times m_{\text{KL}}}$ and $\boldsymbol{\xi} \in \mathbb{R}^{m_{\text{KL}}}$. Here, $m_{\text{KL}}$ is the number of modes we include in the expansion, and $n$ is the number of discretisation points of the mesh we are solving the PDE on. On the other hand, the circulant embedding method is inherently discrete, and so for each PDE solve we need only compute the FFT of $\boldsymbol{\Lambda} \odot \boldsymbol{\xi}$, where $\boldsymbol{\Lambda}, \boldsymbol{\xi} \in \mathbb{R}^s$, with $s = 4n^2$ generally, which is considerably faster. This makes the computational cost diminish tenfold when using MLMC with standard circulant embedding compared to using a uniformly truncated KL-expansion, particularly for large accuracies.

\section{Conclusions and Outlook}
 \label{sec: conclusions}

Circulant Embedding methods are computationally efficient methods for sampling from random fields, as they exploit the Fast Fourier Transform algorithm to achieve speed-up. In the context of uncertainty quantification using Multilevel Monte Carlo for PDEs with random coefficients, this can be applied to generate realisations from the random parameter. 

In some applications, such as groundwater flow, we are interested in the case when the covariance function associated with the random parameter has a short correlation length relative to the size of the computational domain, so that realisations exhibit small-scale fluctuations. In such a case, it is not possible to use very coarse meshes in the MLMC estimator, which limits the subsequent number of levels we can exploit. {Hence, this renders the standard MLMC method computationally expensive}. 

The circulant embedding with smoothing technique developed here  alleviates this issue by producing approximate samples from the random field, which allows to freely choose the mesh size used in the coarsest level approximation of MLMC. {By exploiting these cheap approximations computed with a large mesh size, we do not modify the asymptotic behaviour of the cost, but we reduce the constant in the complexity result by up to a factor of 5-10}. This paper looks not only at the theoretical analysis of the approximation error {introduced by smoothing the random field samples}, but also showcases numerical results which illustrate the computational savings that the smoothing technique can achieve in practical examples.

A straightforward extension is to integrate the proposed approach with MLQMC. As QMC estimators are known to outperform standard MC, we expect that this would yield a more efficient method. In addition, the smoothing technique is similar in nature to the level-dependent truncation of the KL-expansion. Analogous extensions could be developed for the spectral representation method \citep{rice_mathematical_1944, shinozuka_digital_1972}, the wavelet reconstruction approach \citep{zeldin_random_1996} or the white noise sampling scheme \citep{lindgren_explicit_2011}.

\section*{Acknowledgements}
AI was supported by the EPSRC Centre for Doctoral Training in Mathematical Modelling, Analysis and Computation (MAC-MIGS) funded by the UK Engineering and Physical Sciences Research Council (grant EP/S023291/1), Heriot-Watt University and the University of Edinburgh. This work used the Cirrus UK National Tier-2 HPC Service at EPCC (http://www.cirrus.ac.uk) funded by the University of Edinburgh and EPSRC (EP/P020267/1). AI and ALT would like to thank the Isaac Newton Institute for Mathematical Sciences, Cambridge, for support and hospitality during the programme {\em Mathematical and statistical foundation of future data-driven engineering} where work on this paper was undertaken. This work was supported by EPSRC grant no EP/R014604/1.

\bibliography{references}

\end{document}